\documentclass[1p]{elsarticle}

\usepackage{amsmath, amssymb, amsthm,  epsfig}
\usepackage{amsfonts}
\usepackage{color}

\theoremstyle{plain}
\newtheorem{theo}{Theorem}[section]
\newtheorem{lem}[theo]{Lemma}
\newtheorem{prop}[theo]{Proposition}
\newtheorem{cor}[theo]{Corollary}
\newtheorem{defi}[theo]{Definition}
\newtheorem{rem}[theo]{Remark}

\numberwithin{equation}{section}

\newcommand{\R}{\mathbb R}
\newcommand{\N}{\mathbb N}
\newcommand{\Heis}{\mathbb H}
\newcommand{\g}{\mathfrak g}

\newcommand{\e}{\epsilon}

\newcommand{\C}{\mathcal C}
\newcommand{\M}{\mathcal M}
\newcommand{\G}{\mathcal G}
\newcommand{\T}{\mathcal T}

\newcommand{\A}{\mathcal A}

\newcommand{\ie}{\textit{i.e.\ }}

\newcommand{\p}{\partial}

\newcommand{\loc}{\text{loc}}

\newcommand{\sign}{\text{sign}}
\newcommand{\sq}{1+|\nabla_\epsilon u|^2}

\newcommand{\tx}{\tilde X_1}
\newcommand{\ty}{\tilde X_2}
\newcommand{\ts}{\tilde X_3}
\newcommand{\Io}{{I_{{}_0}}}
\newcommand{\Jo}{{J_{{}_0}}}
\newcommand{\eo}{{e_{{}_0}}}
\newcommand{\Xo}{{X_{{}_0}}}
\newcommand{\Yo}{{Y_{{}_0}}}
\newcommand{\IIo}{{\mathcal{I}_{{}_0}}}

\newcommand{\No}{{N_{{}_0}}}
\newcommand{\Go}{{\Gamma_{{}_0}}}
\newcommand{\no}{{\nabla_{{}_0}}}
\newcommand{\de}{{d_{{}_0}}}
\newcommand{\Bo}{{B_{{}_0}}}
\newcommand{\ao}{{a_{{}_0}}}

\begin{document}

\title{Regularity of minimal intrinsic graphs in 3 dimensional sub-Riemannian structures of step 2}

\author[dm]{D. Barbieri\fnref{fn1}}
\ead{barbier@dm.unibo.it}

\author[dm]{G. Citti\fnref{fn1}}
\ead{citti@dm.unibo.it}

\fntext[fn1]{The authors have been partially supported by the European Projects GALA (FP6) and CG DICE (FP7).}

\address[dm]{Universit\`a di Bologna, Dipartimento di Matematica, p.zza di p.ta S.Donato 5, 40126 Bologna, Italy.}

\begin{abstract}
This work provides a characterization of the regularity of
noncharacteristic intrinsic minimal graphs for a class of vector fields
that includes non nilpotent Lie algebras as the one given by
Euclidean motions of the plane. The main result extends a previous
one on the Heisenberg group, using similar techniques to deal with
nonlinearities. This wider setting provides a better understanding
of geometric constraints, together with an extension of the
potentialities of specific tools as the lifting-freezing procedure
and interpolation inequalities. As a consequence of the regularity,
a foliation result for minimal graphs is obtained.

\vspace{10pt}\noindent
{\bf Resum\'e}

\vspace{6pt}\noindent
Ce travail fournit une caract\'erisation de la r\'egularit\'e des graphiques
intrins\`eques minimaux pour une classe de champs de vecteurs qui comprend
alg\`ebres de Lie non nilpotent comme celles obtenues dans le groupe des
mouvements rigides du plan Euclidien. Le r\'esultat principal \'etend un
pr\'ec\'edent r\'esultat dans le groupe d'Heisenberg, en utilisant des
techniques similaires pour traiter les non-lin\'earit\'es. Dans ce cadre plus
g\'en\'eral, on peut mieux comprendre les contraintes g\'eom\'etriques et les
potentialit\'es des outils sp\'ecifiques comme la proc\'edure de lifting-freezing
et les in\'egalit\'es d'interpolation. A la fin un r\'esultat de foliation
pour les graphiques minimaux est obtenu comme une cons\'equence de la r\'egularit\'e.

\end{abstract}


\maketitle

\section{Introduction}

We prove a regularity result for noncharacteristic intrinsic
minimal graphs in a 3 dimensional sub-Riemannian
contact structure. Such a sub-Riemannian structure is defined by the
choice at every point of $\R^3$ of a couple of vector fields $X_1,
X_3$, which, together with their first order commutator
\begin{equation}\label{vectorfields}
X_2 = [X_1,X_3]
\end{equation}
span the tangent space at any point.
This condition is a special case of the well known  H\"ormander type
condition \cite{hormander} and ensures that it is possible to endow
the space with a sub-Riemannian metric (see \cite{Bellaiche}).
However we will see that low dimensional situation is particularly
interesting for our problem, since minimal graphs in 3 dimensional
Lie group satisfy a foliation result, which is not present in higher
dimension.

As it is usual, we will call horizontal tangent bundle the bundle
$\mathbb H$ spanned at every point by $X_1$ and $X_3$, and we will
define on it a metric $g_0$, by requiring that $\{X_1,X_{3}\}$ form
a $g_0-$orthonormal basis of $\mathbb H$. For each $\e>0$ we extend
$g_0$ to a Riemannian metric $g_\e$ by requiring that $\{X_1,X_{3},
X_{2}\}$ form a $g_\e-$orthonormal basis of $\mathbb H$.

The notion of regular surface in the
sub-Riemannian setting has been introduced by Franchi, Serapioni
and Serra Cassano, who (in \cite{FSSC}, \cite{FSSC4})  extended  to this
context the celebrated blow up technique introduced by De Giorgi.
They defined a regular surface $M$ as the zero level set of a
function $f$, whose intrinsic gradient $(X_1 f, X_3 f)$ never
vanishes. The horizontal tangent space to $M$, denoted by $H\M$, is
the fiber bundle whose fibers $H_p\M$ at any point $p$ are given by
the intersection between the Euclidean tangent space $\T_p\M$ and
the horizontal tangent space $\mathbb H_p$. The assumption of nonvanishing gradient on $f$
ensures that $H_p\M$ is spanned at
every point by a single vector field, called horizontal tangent
vector to $M$ at the point $p$. Hence a surface of this type is not
regular in the standard Euclidean sense and in general it will be a
fractal (see \cite{KS}). In particular, the standard implicit function result
does not hold in this case. However,
up to a suitable change of coordinates, we can assume that the
vector fields $X_i$ are represented as
$$
X_3 = \partial_{x_3} \, , \quad X_i = \sigma_i^j(x_1,x_2,x_3)\p_{x_j} \ \ i = 1, 2
$$
and in these coordinates the surface $M$ becomes an intrinsic graph of the form
\begin{equation}\label{intrinsicgraph} \M = \{x_3 = u(x_1,x_2);
(x_1,x_2) \in \Omega \subset \R^2\}\ .
\end{equation}
It has been proved in \cite{ASCV} and \cite{Implicit} that the function
$u$ is differentiable along the vector field $X_{1,u}$, projection
on $\T_\Omega$ of the horizontal tangent vector to $M$:
$$
X_{1,u} = X_1\big|_\M = \sigma_1^1(x_1,x_2,u(x_1,x_2))\p_{x_1} +
\sigma_1^2(x_1,x_2,u(x_1,x_2))\p_{x_2}\ .
$$
Besides, if the Lie derivative $X_{1,u}u$ is continuous, the function $u$ is
called of class $\C^1_u$. In the particular case of the Heisenberg group a
systematic study of properties of spaces $\C^1_u$ has been carried
out in \cite{FSSC3}, \cite{FSSC2}, \cite{fsc-bigolin}, \cite{Bigolinserracassano}.

\subsection{Minimal surfaces}

The first definition of minimal surfaces in this context is due to
Garofalo and Nhieu \cite{GN}, as first variation of the perimeter
functional. Pauls in \cite{pau:minimal}, studied the same notion as
the elliptic regularization of the sub-Riemannian metric. A
different, but equivalent notion was given by Cheng, Hwang, Malchiodi and Yang \cite{CHMY}
in general CR manifolds. In particular the specific case of three
dimensional pseudo-hermitian manifolds
has been outlined in \cite{Malchiodi}. Later, Danielli, Garofalo and Nhieu in \cite{dgn:minimal} extended these concepts with a general notion of horizontal mean curvature in sub-Riemannian setting. Ritor\'e and Rosales, in \cite{RR1} and \cite{RR2}, use a different approach
in order to introduce a notion of constant mean curvature under a volume constraint. We also refer to \cite{montefalcone},
\cite{Shcherbakova} for different notions of first variation of the
perimeter, and to \cite{cdpt:survey} for a survey of the problem.

The condition of zero horizontal mean curvature can be expressed as
follows in terms of the previously introduced vector field:
\begin{equation}\label{equation}
X_{1,u} \left(\frac{X_{1,u} u}{\sqrt{1+|X_{1,u}u|^2}}\right) = 0
\end{equation}
(see \cite{CCM1},  \cite{bascv},  \cite{dgn2}). Indeed equation (\ref{equation}) is the Euler-Lagrange equation for the sub-Riemannian area functional
$$
Lip(\Omega) \ni u \mapsto P_{X_1,X_3}(E_u,\Omega\times\R) = \int_\Omega \sqrt{1 + |X_{1,u}u|^2} dx_1dx_2
$$
where $E_u\doteq \{(x_1,x_2,x_3) \in \Omega\times\R \, : \, x_3 < u(x_1,x_2); (x_1,x_2) \in \Omega\}$ and $P_{X_1,X_3}(E_u,\Omega\times\R)$ denotes the sub-Riemannian perimeter measure of $E_u$ in $\Omega\times\R$ associated to the system of vector fields $X_1$ and $X_3 \equiv \partial_{x_3}$, according to the definition introduced in \cite{GN}.

It is also well
known (at least in the examples of the Heisenberg group and the
rototranslation group, see \cite{CHMY}, \cite{CS}, \cite{HP2008})
that for regular curves the mean curvature coincides with the
curvature of the projection of the Legendrian leaves on the
horizontal plane. However the same result is know for Lipschitz
continuous solution only in the Heisenberg case.

On the other side weaker definitions of minimal surfaces have been
proposed (see \cite{ASCV}, \cite{CHMY}, \cite{ChengHwangYang},
\cite{chy}, \cite{GN}, \cite{pau:obstructions}, \cite{pau:minimal}).
Explicit examples of nonregular surfaces (in the Euclidean sense)
have been provided by \cite{pau:obstructions}, \cite{pau:minimal},
by \cite{R} and in \cite{chy}. Hence the problem of regularity is
particular interesting, and by now only a few results are known in
the Heisenberg group \cite{CCM1}, \cite{fsc-bigolin}.

Here we plan to study regularity of Lipschitz continuous solutions
in general 3D structures. Of course we try to adapt to this
situation the typical instruments of nonlinear PDEs. The celebrated
works of H\"ormander \cite{hormander} and Rothschild and Stein
\cite{Roth:Stein} started the study of linear operators defined in
terms of vector fields, satisfying the bracket generating condition.
We refer to the monographs \cite{BLU} and
\cite{VSC} for an exhaustive presentation of the theory, which
however, can not be applied here for two main reasons. The first one
is that our vector fields are nonlinear, since they explicitly
depend on the solution. To our knowledge the only known results
regarding H\"ormander type nonlinear vector fields have been
obtained in \cite{CCMn}, \cite{Wheeden}, while studying curvature
equations in a completely different context. However not even these
more delicate results can be applied to our situation, since we are
dealing with only one vector field in a two dimensional space. Clearly one vector field can not generate the whole tangent space at any point, hence the
H\"ormander condition is violated at every point and because of this situation we can obtain a foliation result.

In order to prove regularity, we will follow
the approach introduced to handle the Levi curvature equation in \cite{CM} or the $\Heis^1$ curvature equation in
\cite{CCM1}. We introduce an elliptic regularization
\begin{equation}\label{elliptic}
L_{\epsilon,u} u = X_{i,u} \left(\frac{X_{i,u} u}{\sqrt{1+|\nabla_\epsilon u|^2}}\right) = 0
\end{equation}
where summation over repeated indices is intended. The regularization is performed adding the transverse direction scaled by a small parameter, so that we denote
$$
X_{2,u} = \epsilon \Big[\sigma_2^1(x_1,x_2,u(x_1,x_2))\p_{x_1} + \sigma_2^2(x_1,x_2,u(x_1,x_2))\p_{x_2}\Big] = \epsilon X_2\big|_{\M}
$$
and $\nabla_\epsilon$ stands for the gradient $(X_{1,u},X_{2,u})$.
The class of minimal graphs under study is then that of vanishing
viscosity solutions:
\begin{defi}\label{defvanishingviscosity}
A function $u \in Lip(\Omega)$ is said to be a vanishing viscosity solution of (\ref{equation}) if there exists a sequence $(\epsilon_n)$ in $\R^+$ with $lim_{n\to\infty}\epsilon_n = 0$ and a sequence $(u_n)$ in $\C^\infty$ such that
\begin{enumerate}
 \item $L_{\epsilon_n,u_n}u_n = 0$ in $\Omega$ for all $n$;
 \item the sequence $(u_n)$ is uniformly bounded in $Lip(\Omega)$, \ie there exists an $l > 0$ such that $\|u_n\|_{Lip(\Omega)}\leq l$, and converges to $u$ uniformly on compact subsets of $\Omega$.
\end{enumerate}
\end{defi}
We note that this approach will allow us, for all positive values of
$\epsilon$, to work with smooth solutions, due to the uniform
ellipticity of (\ref{elliptic}). This point will be of key
importance since it permits to skip passing through density of
smooth functions, which is not guaranteed in this general setting.

\subsection{Main results}
Throughout all the paper we will assume the solution $u$ to be
Euclidean Lipschitz, \ie we set
$$
Y_u = X_2\big|_{\M} = \frac{1}{\epsilon}X_{2,u}
$$
and assume that
\begin{equation}\label{assumption}
\|X_{1,u}u\|_{L^\infty(\Omega)} + \|Y_{u}u\|_{L^\infty(\Omega)} \leq M < \infty\ .
\end{equation}

In order to adopt the previously described approach, and hence work with smooth approximating functions, we will also require that the coefficients $\sigma_i^j$ be smooth, more precisely we will intend them as $\C^{\infty}$ functions.

Under this assumption we will prove the following result
\begin{theo}\label{theoCregularity}
Let $u$ be a vanishing viscosity solution of equation
(\ref{equation}). Then for all $k > 0$ and all $\alpha < 1$ it holds
\begin{equation}\label{Cregularity}
X_{1,u}^ku \in \C^\alpha_{loc}(\Omega)\ .
\end{equation}
\end{theo}

The main regularity result, expressed by Theorem
\ref{theoCregularity}, provides $\C^\infty$ regularity along the
horizontal direction. This in particular implies that equation
(\ref{equation}) can be represented simply as $X_{1,u}^2u=0$
pointwise everywhere in $\Omega$. The geometric meaning is expressed
by Corollary \ref{CORfoliation}, which ensure that minimal graphs
are foliated into horizontal curves, along which they are smooth.

\begin{cor}\label{CORfoliation}
Let $x_3=u(x_1,x_2), (x_1,x_2) \in \Omega$ be a vanishing viscosity
minimal graph. Then the flow of the vector field $X_{1,u}$ yields a
foliation of the domain $\Omega$ by horizontal curves $\gamma$.
More precisely, let $(x_1,x_2)\in\Omega$ and denote by
$\gamma$ the integral curve of $X_{1,u}$ passing through it. Then $u$ is
differentiable at $(x_1,x_2)$ along $\gamma$ and equation
(\ref{equation}) reduces to
$$
\frac{d^2}{dt^2} u(\gamma(t)) = 0\ .
$$
\end{cor}
We remark that no regularity can be achieved in the transverse
direction, since this does not play a role in the intrinsic curvature
equation.

This result then extends to a very general setting results only
known in the special case of the Heisenberg group. We explicitly
note that the largest majority of the results known for intrinsic
graphs only apply to nilpotent or homogeneous groups, while here this
condition is not required. In this wider setting only the definition
of regular surface is known. Adapting the proof in this contest is
not purely an extension, but also provides a better understanding of
the role of geometric constraints in the present approach to
surfaces regularity, which in the end turns out not to rely on any
nilpotency assumption.

We explicitly note that results of this type are not expected in
higher dimension. In the specific case of the Heisenberg group, it
is indeed known that 2-dimensional minimal surfaces foliates in regular
curves, and are not regular as surfaces. In higher dimension the
surfaces are $\C^{\infty}$ in the Euclidean sense and there is no
foliation result.

\subsection{Sketch of the proof}

We first introduce Sobolev spaces associated to the vector fields
$X_{i,u}$: we say that a function $z$ belong to the Sobolev space $
W_\e^{1,p}$ if $$ z,\nabla_{\e} z\in L^p,$$ and in this case we set
$$\|z\|_{W_\e^{1,p}} =\|z\|_{L^p} + \|\nabla_{\e}z\|_{L^p}.$$
Higher order Sobolev spaces are defined in a similar way. We
explicitly note that this definition contains a linearization, since
we now apply the vector fields $X_{i,u}$ to any function $z$, not
necessarly to the function $u$.

Accordingly, in Section 3 we study the regularity properties of
solutions for such a linearization of equation (\ref{elliptic}).
This is done applying a Moser-like iteration procedure, and to this end
we have worked with the linear operator (with nonregular coefficients) that defines the equation satisfied by the
derivatives of the solution to the linearized equation, namely
$$
M_{\epsilon,u}z = X_{i,u} \left(\frac{a_{ij}(\nabla_\epsilon
u)}{\sqrt{1+|\nabla_\epsilon u|^2}}X_{j,u}z\right) = f
$$
where the coefficients $a_{ij}$ are regular and uniformly elliptic
(see also definition (\ref{coefficients})).

The Moser iteration technique has been extended for equation expressed
in terms of vector fields, which satisfy Sobolev embedding theorem.
This theorem is in general true for H\"ormander vector fields, with
regular or nonregular coefficients (see e.g. \cite{J}, \cite{CDG}, \cite{HK} and references therein).
However we can not hope to prove a Sobolev inequality uniform in $\e$, since the
equation does not naturally involve derivatives in the transverse
direction: indeed the intrinsic geometry of these minimal graphs will
induce a Legendrian foliation, so one cannot expect to gain Sobolev
regularity when passing from one fiber of the foliation to the
other. Hence following {\cite{CMa} and \cite{CCM1} we replace these
type of inequalities with an interpolation inequality, Proposition
\ref{PROPinterp1}. Another crucial ingredient in the Moser method
is a  Caccioppoli type inequality. This is then obtained in Lemma
\ref{LEMcaccioppolisecond}. Combining these two instruments with a
bootstrap argument, we obtain estimates for a solution of the linear
equation in terms of the coefficients of the equation itself in Theorem \ref{THEOiterations}.

Unfortunately the regularity for the coefficients required by the cited theorem is not
available under assumption (\ref{assumption}). For this reason we prove in section 4 a
rather delicate Sobolev type estimate to start the iteration. This estimate will not be
optimal, since for the previously recalled reason we can not hope to obtain an
intrinsic Sobolev inequality with optimal exponent. Moreover,
we will need assumptions both on the intrinsic derivatives $X_{1,u}u$ and on the
transverse derivatives in direction $Y_u$. The
main idea is to approximate, through a freezing and a lifting
procedure, the vector fields $X_{i,u}$ with H\"ormander type vector
fields of step three. We will choose the approximating vector fields
in such a way that they do not span the space at step 2, so that
the second derivatives along the approximating operator provide a
good approximation of the corresponding derivatives along the
vectors $X_{i,u}$. On the other side, since the approximating vector
fields satisfy the H\"ormander condition, we will be able to establish
the required Sobolev inequality. This will be done in Theorem \ref{THEOfreesing}. We
explicitly note that this freezing and lifting method is only
vaguely reminiscent of the celebrated method of Rothschild, and
Stein \cite{Roth:Stein}. Indeed they approximate H\"ormander
operators with nilpotent operators with the same bracket generating
property. Here on the contrary the novelty of the idea is to adapt
this method to non-H\"ormander vectors fields.

Finally in section 5 we obtain the regularity in the Sobolev spaces
for viscosity solutions of the nonlinear equation (\ref{equation}),
concluding the proof of Theorem \ref{theoCregularity}.

\section{Notations and technical facts}

\subsection{Significative examples of such structures}

This setting include all three dimensional simply connected Lie
group $\G$ whose Lie algebra $\g$ can be generated by two vector
fields. The simplest noncommutative Lie algebra with two generator
is the three dimensional Heisenberg group, whose left invariant
vector fields over $\R^3$ read
\begin{displaymath}
X_1 = \partial_{x_1} + x_3 \partial_{x_2} , \ X_2=
\partial_{x_2}, \ X_3 =
\partial_{x_3}.
\end{displaymath}

A relevant example of non nilpotent Lie group is given by the group
$E(2)$ of translations and rotations of the Euclidean plane. In this
case the vectors $X_1$ and $X_2$ express translations in two
orthogonal directions of $\R^2$ while $X_3$ describes a rotation.
Hence they can be expressed as vector fields over $\R^2_{x_1, x_2}
\times S^1_{x_3}$ with components
\begin{displaymath}
X_1 = \cos(x_3) \partial_{x_1} + \sin(x_3) \partial_{x_2}, \ X_2=
-\sin(x_3) \partial_{x_1} + \cos(x_3) \partial_{x_2}, \ X_3 =
\partial_{x_3}.
\end{displaymath}

And the bracket relations are the following:
\begin{displaymath}
[X_1,X_3] = - X_2 \ , \quad [X_1,X_2] = 0 \ , \quad [X_2,X_3] = X_1
\ .
\end{displaymath}
This case is particularly important since the related minimal
surfaces provide a concrete model for the geometric completion
operated by the primary visual cortex, as introduced in \cite{CS}.
Indeed an interesting perspective on minimal surfaces in this
setting was recently obtained in \cite{HP2009}.

Other notable Lie group included in the present investigation are
$SO(3)$, or $SL(2)$, but we remark that the group structure is not
an essential point in defining the setting.

\subsection{Linear vector fields in the sub-Riemannian space}
Let us first recall some properties of the linear vector fields
$X_i$, defined in (\ref{vectorfields}) on $R^3$. In order to obtain
the asserted representation of the vector field $X_3$ as a partial
derivative, we need to introduce polarized exponential coordinates.

Indeed  to any pair of points $(x,y,z),
(\overline{x},\overline{y},\overline{z}) \in \R^3$ sufficiently
close we can associate coordinates $(x_1,x_2,x_3) \in \R^3$ such
that
\begin{equation}\label{polarized}
(x,y,z) = \exp(x_3X_3) \exp(x_1X_1 +
x_2X_2)(\overline{x},\overline{y},\overline{z}).
\end{equation}
Relation (\ref{polarized}) simply means that if we take two curves
$\gamma_1,\gamma_2$ in $\R^3$ such that
$$
\gamma_1(0) = (\overline{x},\overline{y},\overline{z}) \ \ , \
\dot{\gamma}_1(t)=x_1X_1(\gamma_1(t)) + x_2X_2(\gamma_1(t))
$$
and
$$
\gamma_2(0) = \gamma_1(1) \ \ , \
\dot{\gamma}_2(t)=x_3X_3(\gamma_2(t))
$$
then, due to the span property of the vector fields, we can always
choose $(x_1,x_2,x_3)$ such that $(x,y,z) = \gamma_2(1)$.

In these coordinates, by a structure result given by
\cite[Lemma 3.1]{Implicit}, we have
\begin{rem}\label{changecoord}
The vector field $X_3$ is expressed in the new coordinates as
$$
X_3 = \p_{x_3}
$$
while for $X_1, X_2$ we can write
\begin{eqnarray*}
X_1 & = & \sigma_1^1(x_1,x_2,x_3)\p_{x_1} + \sigma_1^2(x_1,x_2,x_3)\p_{x_2} + \sigma_1^3(x_1,x_2,x_3)\p_{x_3}\\
X_2 & = & [X_3,X_1] = \sigma_2^1(x_1,x_2,x_3)\p_{x_1} +
\sigma_2^2(x_1,x_2,x_3)\p_{x_2} +  \sigma_2^3(x_1,x_2,x_3)\p_{x_3}
\end{eqnarray*}
where condition (\ref{vectorfields}) sets the $\sigma_2^i$'s
to $\sigma_2^i (x_1,x_2,x_3) = \p_{x_3}\sigma_1^i (x_1,x_2,x_3)$.
\end{rem}

The span condition for the vector fields can be written as
$$
\textrm{rank}\left(
\begin{array}{cc}
\sigma_1^1 & \sigma_1^2\\
\p_{x_3}\sigma_1^1 & \p_{x_3}\sigma_1^2
\end{array}\right) = 2
$$
\ie we assume that
\begin{equation}\label{rank}
\Big|\sigma_1^1(x_1,x_2,x_3)\p_{x_3}\sigma_1^2(x_1,x_2,x_3) -
\sigma_1^2(x_1,x_2,x_3)\p_{x_3}\sigma_1^1(x_1,x_2,x_3)\Big| > 0
\end{equation}
for all $(x_1,x_2,x_3) \in \G$.\\
We remark that we do not assume nilpotency, so that
\begin{eqnarray*}
\ [X_2,X_3] & = & c^1_{2,3} X_1 + c^2_{2,3} X_2 + c^3_{2,3} X_3\\
\ [X_1,X_2] & = & c^1_{1,2} X_1 + c^2_{1,2} X_2 + c^3_{2,3} X_3
\end{eqnarray*}
where the structure coefficients $c^i_{j,k}(x_1,x_2,x_3)$ can well
be nonzero.

\subsection{Projected vector fields in $\R^2$}
For what concerns the nonlinear vector fields $X_{1,u}, X_{2,u}$, their formal
adjoints are given by
\begin{equation}\label{adjoints}
X_{i,u}^\dag = - X_{i,u} - m_i(x_1,x_2)
\end{equation}
where
\begin{eqnarray*}
m_1(x_1,x_2) & = & \p_{x_1}\sigma_1^1(x_1,x_2,u(x_1,x_2)) + \p_{x_2}\sigma_1^2(x_1,x_2,u(x_1,x_2))\\
& = & \p_i \sigma_1^i(x_1,x_2,u(x_1,x_2)) + \p_3 \sigma_1^i(x_1,x_2,u(x_1,x_2)) \p_{x_i}u(x_1,x_2)\\
& = & \p_i \sigma_1^i + [X_3,X_1]\big|_\M u = \p_i \sigma_1^i + \frac{1}{\epsilon}X_{2,u}u\\
m_2(x_1,x_2) & = & \epsilon \left(\p_{x_1}\sigma_2^1(x_1,x_2,u(x_1,x_2)) + \p_{x_2}\sigma_2^2(x_1,x_2,u(x_1,x_2))\right)\\
& = & \epsilon\left(\p_i \sigma_2^i(x_1,x_2,u(x_1,x_2)) + \p_3 \sigma_2^i(x_1,x_2,u(x_1,x_2)) \p_{x_i}u(x_1,x_2)\right)\\
& = & \epsilon\left(\p_i \sigma_2^i + [X_3,X_2]\big|_\M u\right) =
\epsilon\p_i \sigma_2^i - \epsilon c^1_{2,3}X_{1,u}u -
c^2_{2,3}X_{2,u}u
\end{eqnarray*}
where $c^i_{jk}$ stands now and in what follows for
$c^i_{jk}(x_1,x_2,u(x_1,x_2))$ and we note that, in agreement with
the choice of $X_{2,u} = \epsilon X_2\big|_\M$, only $m_2$ is of
global order $\epsilon$.

Moreover, their commutator can be written explicitly:
\begin{eqnarray*}
[X_{1,u},X_{2,u}] & = & \epsilon[\sigma_1^j \p_{x_j}, \sigma_2^k \p_{x_k}] = \epsilon\left(\sigma_1^j (\p_{x_j}\sigma_2^k)\p_{x_k} - \sigma_2^k(\p_{x_k}\sigma_1^j)\p_{x_j}\right)\\
& = & \epsilon\left(\sigma_1^j (\p_{j}\sigma_2^k)\p_{x_k} - \sigma_2^k(\p_{k}\sigma_1^j)\p_{x_j}\right)\\
&& + \ \epsilon\sigma_1^j (\p_3 \sigma_2^k \p_{x_j} u)\p_{x_k} - \epsilon\sigma_2^k(\p_3 \sigma_1^j \p_{x_k}u)\p_{x_j}\\
& = & \epsilon[X_1,X_2]\big|_{\M} + \epsilon\left(X_{1,u}u\right)[X_3,X_2]\big|_{\M} - \left(X_{2,u}u\right)X_2\big|_{\M}\\
& = & \omega^i(x_1,x_2)X_{i,u}
\end{eqnarray*}
where
\begin{displaymath}
\begin{array}{rcl}
\omega^1(x_1,x_2) & = & \epsilon\left(c^1_{1,2} - \left(X_{1,u}u\right)c^1_{2,3}\right)\vspace{0.2cm}\\
\omega^2(x_1,x_2) & = & c^2_{1,2} - \left(X_{1,u}u\right)c^2_{2,3} -
\frac{1}{\epsilon}\left(X_{2,u}u\right)\ .
\end{array}
\end{displaymath}
By $\omega^i_{j,k}$ we will indicate the corresponding commutator
coefficients, antisymmetric with respect to lower indices, \ie
$\omega^i_{j,k} = - \omega^i_{k,j}$ and
\begin{equation}\label{structure}
\omega^i_{j,k}(x_1,x_2) = \left\{
\begin{array}{ll}
0 & \textrm{if} \ j = k\\
\omega^i(x_1,x_2) & \textrm{if} \ j = 1, k = 2
\end{array}\right.
\end{equation}
so that in general we can write
$$
[X_{j,u},X_{k,u}] = \omega^i_{j,k} X_{i,u}
$$
and we note that this commutator is of global order $\epsilon$, in
agreement with the choice of scaling $X_{2,u}$ with $\epsilon$.

This implies in particular that, under hypothesis (\ref{assumption}), there exists a finite po\-si\-tive constant $C_M$
depending only on $M$ and $\Omega$ such that
\begin{equation}\label{assCM}
\|m_1\|_{L^\infty(\Omega)} +
\frac{1}{\epsilon}\|m_2\|_{L^\infty(\Omega)} +
\frac{1}{\epsilon}\|\omega^1\|_{L^\infty(\Omega)} +
\|\omega^2\|_{L^\infty(\Omega)} \leq C_M < \infty\ .
\end{equation}

\section{A priori estimates I: Caccioppoli inequalities}
In this section we provide a priori Sobolev estimates for solutions $u$ to equation (\ref{elliptic}). The adopted procedure relies on the equations satisfied by the derivatives $z_1=X_{k,u}u$, $z_2=X_{l,u}X_{k,u}u$ and $z_Y=Y_u u$ of such solutions: these equations will be treated as linear equations with respect to $z$, considering $u$ as part of the coefficients of the differential operator that defines the equations. The first two lemmata of the section show that they take the following divergence form
\begin{equation}\label{divergence}
M_{\epsilon,u}z = X_{i,u} \left(\frac{a_{ij}(\nabla_\epsilon u)}{\sqrt{1+|\nabla_\epsilon u|^2}}X_{j,u}z\right) = f
\end{equation}
where the precise formulations of the nonhomogeneous term will be given and the coefficients $a_{ij}$ are
\begin{equation}\label{coefficients}
a_{ij}(\nu) = \delta_{ij} - \frac{\nu_i\nu_j}{1+|\nu|^2}
\end{equation}
and we note at they are symmetric, uniformly bounded and positive definite, \ie
\begin{equation}\label{positive}
\frac{1}{1+|\nu|^2}|\xi|^2 \leq a_{ij}(\nu) \xi_i \xi_j \leq |\xi|^2 \quad \forall \ \xi \in \R^2 \ , \ \forall \ \nu \in \R^2\ .
\end{equation}
The final result will be the following

\begin{theo}\label{THEOiterations}
Let $u$ be a function satisfying (\ref{assumption}). Let $z$ be a smooth solution of equation (\ref{divergence}) and $f$ be a given $\C^{\infty}(\Omega)$ function, then for any $p \geq 3$ there exists a positive constant $C$ depending only on $p,\Omega$ and the constant $M$ of (\ref{assumption}) but independent of $\epsilon$ such that for any $m\geq 1$ it holds
\begin{eqnarray*}
i) & \displaystyle{\|z\|_{W_\epsilon^{m+1,p+1/2}(\Omega)}^{p+1/2}} & \leq \ C\left(1 + \|z\|_{W_\epsilon^{m,4p+2}(\Omega)}^{4p+2} + \|u\|_{W_\epsilon^{m+1,(4p+2)/3}(\Omega)}^{(4p+2)/3}\right.\\
&& \!\!\!\!\!\!\!\!\!\!\!\!\!\!\!\!\!\!\!\!\!\!\!\!\!\!\!\!\!\!\!\!\!\!\! + \left. \|u\|_{W_\epsilon^{m+2,(2p+1)/3}(\Omega)}^{(2p+1)/3} + \|Yu\|_{W_\epsilon^{m,(2p+1)/3}(\Omega)}^{(2p+1)/3} + \|f\|_{W_\epsilon^{m,(4p+2)/7}(\Omega)}^{(4p+2)/7}\right)\\
ii) & \displaystyle{\| |\nabla_\epsilon^{m+1} z|^{(p-1)/2} \|_{W_\epsilon^{1,2}(\Omega)}^{2}} & \leq \ C\left(1 + \|z\|_{W_\epsilon^{m,4p+2}(\Omega)}^{4p+2} + \|u\|_{W_\epsilon^{m+1,(4p+2)/3}(\Omega)}^{(4p+2)/3}\right.\\
&& \!\!\!\!\!\!\!\!\!\!\!\!\!\!\!\!\!\!\!\!\!\!\!\!\!\!\!\!\!\!\!\!\!\!\! + \left. \|u\|_{W_\epsilon^{m+2,(2p+1)/3}(\Omega)}^{(2p+1)/3} + \|Yu\|_{W_\epsilon^{m,(2p+1)/3}(\Omega)}^{(2p+1)/3} + \|f\|_{W_\epsilon^{m,(4p+2)/7}(\Omega)}^{(4p+2)/7}\right)
\end{eqnarray*}
\end{theo}

For the sake of simplicity, from now on we will abandon the subscript $u$ in the nonlinear vector fields, and we will make use of the following short-hand notation
$$
\A_{ij} (\nu) = \frac{a_{ij}(\nu)}{\sqrt{1+|\nu|^2}}
$$
so that we can write the operator in equation (\ref{divergence}) as $M_{\epsilon,u} = X_i \A_{ij}(\nabla_\epsilon u)X_j$, or
$$
M_{\epsilon,u} = X_i \A_{ij}X_j\ .
$$

By differentiating equation (\ref{elliptic}), we have the following lemmata.

\begin{lem}\label{LEMfirstderivatives}
If $u$ is a smooth solution of $L_{\epsilon,u} u = 0$, then $z=X_k u$ is a solution of equation
\begin{equation}\label{diveq1}
M_{\epsilon,u} z = f
\end{equation}
where the term at the right hand side is given by
\begin{displaymath}
f = - X_i \A_{ij}\omega^l_{k,j}X_l u - \omega^l_{k,i}X_l\left(\frac{X_i u}{\sqrt{1+|\nabla_\epsilon u|^2}}\right)
\end{displaymath}
and the constants $\omega^l_{j,k}$ are given by (\ref{structure}).
\end{lem}
\begin{proof}
\begin{eqnarray*}
0 & = & L_{\epsilon,u} u = X_k L_{\epsilon,u} u = X_k X_i \left(\frac{X_i u}{\sqrt{\sq}}\right)\\
& = & [X_k, X_i] \left(\frac{X_i u}{\sqrt{\sq}}\right) + X_i X_k \left(\frac{X_i u}{\sqrt{\sq}}\right)\\
& = & [X_k, X_i] \left(\frac{X_i u}{\sqrt{\sq}}\right) + X_i \left[\frac{X_i X_k u}{\sqrt{\sq}} + \frac{[X_k, X_i] u}{\sqrt{\sq}}\right.\\
&& \left. - (X_i u)\left(\frac{(X_j u) (X_j X_k u + [X_k,X_j]u)}{\sqrt{\sq}^3}\right)\right]\\
& = & [X_k, X_i] \left(\frac{X_i u}{\sqrt{\sq}}\right) + X_i \left[\frac{X_i z}{\sqrt{\sq}} - \frac{(X_i u)(X_j u)}{\sqrt{\sq}^3} X_j z\right]\\
&& + X_i\left[\frac{[X_k,X_i]u}{\sqrt{\sq}} - \frac{(X_i u)(X_j u)}{\sqrt{\sq}^3}[X_k,X_j]u\right]\ .
\end{eqnarray*}
\end{proof}

\begin{lem}\label{LEMsecondderivatives}
If $z$ is a smooth solution of $M_{\epsilon,u} z = f$, then $s=X_k z$ is again a solution of the same equation, but with a different right hand side, \ie
\begin{equation}\label{diveq2}
M_{\epsilon,u} s = \rho
\end{equation}
where $\rho$ is given by
\begin{eqnarray*}
\rho & = & X_k f - \left(X_iX_k \A_{ij}\right)\left(X_j z\right) - X_i \A_{ij}\omega^l_{k,j}X_l z - \omega^l_{k,i}X_l \A_{ij} X_j z
\end{eqnarray*}
and the constants $\omega^l_{j,k}$ are given by (\ref{structure}).
\end{lem}
\begin{proof}
\begin{eqnarray*}
X_k f & = & X_k X_i \A_{ij} X_j z = X_i X_k \A_{ij} X_j z + [X_k, X_i] \A_{ij} X_j z\\
& = & X_i\left[(X_k\A_{ij})(X_jz) + \A_{ij}X_jX_kz + A_{ij}[X_k,X_j]z\right] + [X_k, X_i] \A_{ij} X_j z\ .
\end{eqnarray*}
\end{proof}

By direct computation and noting that, under hypothesis (\ref{assumption}), there exists a positive constant $C$ depending only on $M$ of (\ref{assumption}) such that
\begin{equation}\label{XA}
|X_i\A_{ij}| \leq C |\nabla^2_\epsilon u| \quad j=1,2
\end{equation}
we obtain the following corollaries.

\begin{cor}\label{CORfz}
There exists a positive constant $C$ depending only on $M$ such that the function $f$ of Lemma \ref{LEMfirstderivatives} satisfies
\begin{displaymath}
|f| \leq C (|\nabla_\epsilon u| + |\nabla_\epsilon^2 u|).
\end{displaymath}
\end{cor}

\begin{cor}\label{CORfv}
If $u$ is a smooth solution of $L_{\epsilon,u} u = 0$, then $v = Y u$ is a solution of equation
$$
M_{\epsilon,u}v = f_v
$$
where $f_v$ is such that there exists a positive constant $C$ depending only on $M$ such that
$$
|f_v| \leq C(1 + |\nabla_\epsilon^2 u| + |\nabla_\epsilon v|)\ .
$$
\end{cor}

From the study of equations (\ref{diveq1}) and (\ref{diveq2}) we can
obtain estimates for the derivatives of the solution to equation
(\ref{elliptic}).

\begin{prop}[First interpolation inequality]\label{PROPinterp1}
Let us assume (\ref{assumption}). Then for every $p \geq 3$ there exists a $\delta^* > 0$ and a positive constant $C$ depending only on $p$, $\Omega$ and $M$ of (\ref{assumption}) such that for every $0 < \delta < \delta^*$ it holds
\begin{displaymath}
\int |X_iz|^{p+1/2}\phi^{2p} \leq \frac{C}{\delta} \left(\int |z|^{4p+2}\phi^{2p} + 1\right)+  \delta \int \left|\nabla_\epsilon\left(|X_iz|^{(p-1)/2}\right)\right|^2\phi^{2p}
\end{displaymath}
for every function $z \in \C^\infty(\Omega)$ and every cutoff
function $\phi$.
\end{prop}
\begin{proof}
In this proof we will not make use of the convention of summation over repeated indices.
\begin{eqnarray*}
I & = & \int |X_iz|^{p+1/2}\phi^{2p} = \int X_iz|X_iz|^{p-1/2}\sign(X_iz)\phi^{2p}\\
& = & \int z X_i^\dag \left[|X_iz|^{p-1/2}\sign(X_iz)\phi^{2p}\right]\\
& = & - \int m_i z |X_iz|^{p-1/2}\sign(X_iz)\phi^{2p} - (p-1/2)\int z X_i^2z|X_iz|^{p-3/2}\phi^{2p}\\
&& - 2p \int z|X_iz|^{p-1/2}\sign(X_iz)\phi^{2p-1}X_i\phi \leq C_M I_1 + (p-1/2) I_2 + 2p I_3
\end{eqnarray*}
where, by Young inequality and noting that
\begin{equation}\label{leibniz}
\int \left|X_i\left(|X_iz|^{(p-1)/2}\right)\right|^2\phi^{2p} = \frac{(p-1)^2}{4} \int |X_i^2z|^2 |X_iz|^{p-3}\phi^{2p}
\end{equation}
we have
\begin{eqnarray*}
I_1 \!\!\!\!& = & \!\!\!\!\int |z| |X_iz|^{p-1/2}\phi^{2p} \leq \delta I + \frac{1}{\delta} \int |z|^{4p+2}\phi^{2p}  + \frac{1}{\delta} \int \phi^{2p}\\
I_2 \!\!\!\!& = & \!\!\!\!\int |z| |X_i^2z||X_iz|^{p-3/2}\phi^{2p} \leq \frac{\delta}{p-1/2} \int \left|X_i\left(|X_iz|^{(p-1)/2}\right)\right|^2\phi^{2p}\\
&& + \delta \frac{(p-1)^2}{4(p-1/2)}I + \frac{1}{\delta} \frac{4(p-1/2)}{(p-1)^2} \int |z|^{4p+2}\phi^{2p}\\
I_3 \!\!\!\!& = & \!\!\!\!\int |z||X_iz|^{p-1/2}\phi^{2p-1}|X_i\phi| \leq \delta I + \frac{1}{\delta} \left(\int |z|^{4p+2}\phi^{2p} + \int |X_i\phi|^{2p} + \int \phi^{2p} \right)\ .
\end{eqnarray*}
Thus
\begin{eqnarray*}
I & \leq & \left(C_M + 2p + \frac{(p-1)^2}{4}\right)\delta I + \frac{1}{\delta} \left( C_M + 2p + \frac{4(p-1/2)^2}{(p-1)^2} \right) \int |z|^{4p+2}\phi^{2p}\\
&& + \delta \int \left|\nabla_\epsilon\left(|X_iz|^{(p-1)/2}\right)\right|^2\phi^{2p} + \frac{C_M + 2p}{\delta}\left(\int |X_i\phi|^{2p} + \int \phi^{2p} \right)
\end{eqnarray*}
hence setting $\delta^* = \left(C_M + 2p + \frac{(p-1)^2}{4}\right)^{-1}$ we obtain the statement.
\end{proof}

For solutions to equations (\ref{diveq1}) and (\ref{diveq2}) it is possible to estimate the last term in Proposition \ref{PROPinterp1}.

\begin{lem}\label{LEMcaccioppolifirst}
Let $z$ be a smooth solution of equation (\ref{diveq1}) and let $f$
be locally summable in $\Omega$, then for every $p\geq 3$ there exists
a positive constant $C_1$ depending only on $p$ and the constant $M$
in (\ref{assumption}) such that for any cutoff function $\phi$  it
holds
$$
\int \left|\nabla_\epsilon \left(|z|^{(p-1)/2}\right)\right|^2\phi^{2p} \leq C_1 \left(\frac{1}{\delta}\int |z|^{p-1} (\phi^{2} + |\nabla_\epsilon \phi|^{2})\phi^{2p-2} - \int f|z|^{p-3}z \phi^{2p} \right)\ .
$$
\end{lem}
We explicitly note that we can not estimate the last term with its
absolute value, but we prefer keeping the minus sign, since a very
delicate estimate will be needed in the next lemma.
\begin{proof}
The proof relies essentially on uniform boundedness and ellipticity of the coefficients $a_{ij}$ (\ref{coefficients}). We start by multiplying both members of (\ref{diveq1}) by $|z|^{p-3}z\phi^{2p}$ and integrating over $\Omega$, then integrate by parts using (\ref{adjoints}). At that point we make use of (\ref{positive}) and Lipschitz condition (\ref{assCM}):
\begin{eqnarray*}
\int f |z|^{p-3}z \!\!\! && \!\!\!\!\!\!\!\!\!\!\!\!\!\phi^{2p} \ = \ \int \left(X_i \A_{ij} X_j z\right) |z|^{p-3}z\phi^{2p}\\
& = & - (p-2) \int \A_{ij} (X_j z) (X_i z) |z|^{p-3}\phi^{2p}\\
&& - 2p \int \A_{ij} (X_j z) |z|^{p-3}z (X_i \phi)\phi^{2p-1} - \int \A_{ij} (X_j z) m_i |z|^{p-3}z \phi^{2p}\\
& \leq & -\frac{p-2}{(1+C_M^2)^{3/2}} \int |\nabla_\epsilon z|^2 |z|^{p-3}\phi^{2p}\\
&& + 2p \int |z|^{p-3}z \phi^{2p-1} |\nabla_\epsilon z| \ |\nabla_\epsilon \phi| + \int |z|^{p-3}z \phi^{2p} |\nabla_\epsilon z| \sqrt{m_1^2 + m_2^2}\\
& \leq & \left((2p + \sqrt{2}C_M)\delta-\frac{p-2}{(1+C_M^2)^{3/2}}\right) \int |\nabla_\epsilon z|^2 |z|^{p-3}\phi^{2p}\\
&& + \frac{2p + \sqrt{2}C_M}{\delta} \int |z|^{p-1} (\phi^2 + |\nabla_\epsilon \phi|^2) \phi^{2p-2}
\end{eqnarray*}
where the last transition is Young inequality. By (\ref{leibniz}) and choosing $\delta$ sufficiently small we then end up with the desired claim.
\end{proof}

\begin{lem}\label{LEMcaccioppolisecond}
Let $z$ be a smooth solution of equation (\ref{diveq1}) and let $f$
be locally summable in $\Omega$, then for every $p\geq 3$ there exists
a positive constant $C_2$ depending only on $p$ and the constant $M$
in (\ref{assumption}) such that for any cutoff function $\phi$ it
holds
\begin{eqnarray*}
\int \left|\nabla_\epsilon \left(|\nabla_\epsilon z|^{(p-1)/2}\right)\right|^2\phi^{2p}
& \leq & C_2 \left(\int |\nabla_\epsilon z|^{p+1/2}\phi^{2p} + \int |\nabla_\epsilon^2 u|^{(4p+2)/3}\phi^{2p}\right.\\
&& \left. + \int |\nabla_\epsilon Yu|^{(2p+1)/3} \phi^{2p} + \int |f|^{(4p+2)/7}\phi^{2p} + 1\right)
\end{eqnarray*}
\end{lem}
\begin{proof}
By Lemma \ref{LEMsecondderivatives}, $s = X_k z$ satisfies equation (\ref{diveq2}), so that we can apply Lemma \ref{LEMcaccioppolifirst} to the left hand side:
\begin{eqnarray*}
I & = & \int \left|\nabla_\epsilon \left(|s|^{(p-1)/2}\right)\right|^2\phi^{2p} \leq C_1 \left(\frac{1}{\delta}\int |s|^{p-1} \phi^{2p} + \frac{1}{\delta}\int |s|^{p-1} |\nabla_\epsilon\phi|^{2}\phi^{2p-2}\right.\\
&& \left. - \int \rho|s|^{p-3}s \phi^{2p} \right) = C_1 \left( \frac{J' + J''}{\delta} - J \right)\ .
\end{eqnarray*}
Let us now consider the last term at the right hand side explicitly in terms of $f$
\begin{eqnarray*}
J & = & \int \rho|s|^{p-3}s \phi^{2p} = \int \left(X_k f\right) |s|^{p-3}s \phi^{2p} - \int \left(X_i\left(X_k \A_{ij}\right)\left(X_j z\right) \right)|s|^{p-3}s \phi^{2p}\\
&& - \int \left(X_i\A_{ij}\omega^l_{k,j}X_l z\right)|s|^{p-3}s \phi^{2p} - \int \omega^l_{k,i}\left(X_l \A_{ij} X_j z\right)|s|^{p-3}s \phi^{2p}
\end{eqnarray*}
and integrate by parts
\begin{eqnarray*}
J & = & (p-2)\int \left(X_k \A_{ij}\right)\left(X_j z \right)\left(X_i s\right)|s|^{p-3} \phi^{2p}\\
&& + 2p\int \left(X_k \A_{ij}\right)\left(X_j z \right)|s|^{p-3}s \left(X_i\phi\right)\phi^{2p-1} + \int m_i \left(X_k \A_{ij}\right)\left(X_j z \right) |s|^{p-3}s \phi^{2p}\\
&& + (p-2)\int \A_{ij}\omega^l_{k,j}\left(X_l z\right)\left(X_i s\right)|s|^{p-3} \phi^{2p}\\
&& + 2p\int \A_{ij}\omega^l_{k,j}\left(X_l z\right)|s|^{p-3}s \left(X_i\phi\right)\phi^{2p-1} + \int m_i \A_{ij}\omega^l_{k,j}\left(X_l z\right) |s|^{p-3}s \phi^{2p}\\
&& + (p-2)\int \A_{ij} \left(X_j z\right)\omega^l_{k,i}\left(X_l s\right)|s|^{p-3} \phi^{2p}\\
&& + 2p\int \A_{ij} \left(X_j z\right)\omega^l_{k,i}|s|^{p-3}s \left(X_l\phi\right)\phi^{2p-1} + \int m_l\omega^l_{k,i}\A_{ij} \left(X_j z\right) |s|^{p-3}s \phi^{2p}\\
&& + \int \A_{ij} \left(X_j z\right)\left(X_l\omega^l_{k,i}\right)|s|^{p-3}s \phi^{2p} - (p-2)\int f \left(X_k s\right)|s|^{p-3} \phi^{2p}\\
&& - 2p\int f |s|^{p-3}s \left(X_k \phi\right) \phi^{2p-1} - \int m_k f |s|^{p-3}s \phi^{2p}\ .
\end{eqnarray*}
We will now work to isolate terms of type $\int |\nabla_\epsilon z|^{p+1/2} \phi^{2p}$ or terms that can be reabsorbed into $I, J'$ or $J''$. Indeed $p+1/2$ is the power needed for the estimate as it results from Proposition \ref{PROPinterp1}, and
\begin{eqnarray}\label{tempJJ}
J' + J'' & = & \int |\nabla_\epsilon z|^{p-1} \left(\phi^2 + |\nabla \phi|^2\right) \phi^{2p-2}\nonumber\\
& \leq & \int |\nabla_\epsilon z|^{p+1/2} \phi^{2p} + \int \phi^{2p} + 2^p\int \left(\phi^2 + |\nabla \phi|^2\right)^p
\end{eqnarray}
so the other terms that appear in the statement result by pairing through Young inequality.

First we observe that there exists a positive constant $C_A$ depending only on $M$ such that
\begin{equation}\label{temp1}
\left|X_k \A_{ij}\right|\xi_i\eta_j \leq C_A|\nabla_\epsilon^2 u||\xi|\eta| \quad \forall \ \xi, \eta \in \R^2
\end{equation}
indeed
\begin{eqnarray*}
\left|X_k \A_{ij}\right| \leq |X_k a_{ij}| + a_{ij}|\nabla_\epsilon^2 u| \leq |X_kX_iuX_ju| + 3a_{ij}|\nabla_\epsilon^2 u|\ .
\end{eqnarray*}

Using (\ref{temp1}) and Young inequality we can estimate the first term by
\begin{eqnarray*}
\left|\int \left(X_k \A_{ij}\right)\left(X_j z \right)\left(X_i s\right)|s|^{p-3} \phi^{2p}\right| \leq C_A \int |\nabla_\epsilon^2 u||\nabla_\epsilon s||\nabla_\epsilon z||s|^{p-3} \phi^{2p}\\
\leq C_A\left[ \delta \frac{4}{(p-1)^2} I + \frac{1}{\delta} \left( \int |\nabla_\epsilon z|^{p+1/2}\phi^{2p} + \int |\nabla_\epsilon^2 u|^{(4p+2)/3}\phi^{2p} \right) \right]
\end{eqnarray*}
All the other terms can be treated analogously, making use of assumption (\ref{assCM}). We only turn our attention to the term that needs to be estimated with $Yu$, since it involves derivatives of the commutator coefficients:
\begin{eqnarray*}
\bigg|\!\!\!\!\!\!&\displaystyle\int&\!\!\!\!\!\!\A_{ij} \left(X_j z\right)\left(X_l\omega^l_{k,i}\right)|s|^{p-3}s \phi^{2p}\bigg| \ \leq \ \int |X\omega| |\nabla_\epsilon z|^{p-1}  \phi^{2p}\\
& \leq & \int |\nabla_\epsilon z|^{p+1/2}\phi^{2p} + \int |X \omega|^{(2p+1)/3} \phi^{2p}\\
& \leq & \int |\nabla_\epsilon z|^{p+1/2}\phi^{2p} + 2^{(2p-2)/3}\int |\nabla_\epsilon Y u|^{(2p+1)/3}\phi^{2p}\\
&& + 2^{(2p-2)/3}(|c^1_{2,3}|+|c^2_{2,3}|)^{(2p+1)/3}\left(\int |\nabla_\epsilon^2 u|^{(4p+2)/3}\phi^{2p} + \int \phi^{2p}\right)
\end{eqnarray*}
where we have indicated $X\omega = X_1 \omega^1 + X_2 \omega^2$, and the last transition holds since
\begin{eqnarray*}
|X\omega|^p & = & \left|\epsilon c^1_{2,3}X_1X_1 u + c^2_{2,3}X_2 X_1 u + \frac{1}{\epsilon }X_2 X_2 u\right|^p\\
& \leq & 2^{p-1}\bigg[ (|c^1_{2,3}|+|c^2_{2,3}|)^p |\nabla_\epsilon^2 u|^{p} + |X_2 Y u|^p \bigg]\ .
\end{eqnarray*}

Keeping only the higher power of $f$ in the last terms, we then end up with
\begin{eqnarray*}
I & \leq & C \left(\delta I + \frac{J' + J''}{\delta} + \int |\nabla_\epsilon z|^{p+1/2}\phi^{2p} + \int |\nabla_\epsilon^2 u|^{(4p+2)/3}\phi^{2p}\right.\\
&& \left. + \int |\nabla_\epsilon Y u|^{(2p+1)/3} \phi^{2p} + \int |f|^{(4p+2)/7}\phi^{2p} + \int \phi^{2p} + \int |\nabla_\epsilon \phi|^{2p} \right)
\end{eqnarray*}
where $C$ is a positive constant depending only on $p,M$. By (\ref{tempJJ}), for $\delta$ sufficiently small we end up with the desired claim.
\end{proof}

\begin{theo}\label{THEOCaccioppoli1}
Let $u$ be a function satisfying (\ref{assumption}). Let $z$ be a smooth solution of equation (\ref{divergence}) and $f$ be a given $\C^{\infty}(\Omega)$ function, then for any $p \geq 3$ there exists a positive constant $C$ depending only on $p,\Omega$ and the constant $M$ of (\ref{assumption}), hence independent of $\epsilon$, such that it holds
\begin{eqnarray*}
\|z\|_{W_\epsilon^{1,p+1/2}(\Omega)}^{p+1/2} + \| |\nabla_\epsilon z|^{(p-1)/2} \|_{W_\epsilon^{1,2}(\Omega)}^{2} & \leq & C_3\left(\|z\|_{L^{4p+2}(\Omega)}^{4p+2} + \|u\|_{W_\epsilon^{2,(4p+2)/3}(\Omega)}^{(4p+2)/3}\right.\\
&& \!\!\!\!\!\!\!\!\!\!\!\!\!\!\!\left. + \ \|Yu\|_{W_\epsilon^{1,(2p+1)/3}(\Omega)}^{(2p+1)/3} + \|f\|_{L^{(4p+2)/7}(\Omega)}^{(4p+2)/7} + 1\right)
\end{eqnarray*}
\end{theo}
\begin{proof}
Applying Lemma \ref{LEMcaccioppolisecond} to Proposition \ref{PROPinterp1} we get
\begin{eqnarray*}
\int && \!\!\!\!\!\!\!\!\!\!|\nabla_\epsilon z|^{p+1/2}\phi^{2p} \ \leq \ \frac{C}{\delta} \left(\int |z|^{4p+2}\phi^{2p} + 1 \right) + \delta \ C_2 \int |\nabla_\epsilon z|^{p+1/2}\phi^{2p}\\
&& \!\!\!\! + \ \delta \ C_2 \left(\int |\nabla_\epsilon^2 u|^{(4p+2)/3}\phi^{2p} + \int |\nabla_\epsilon Y u|^{(2p+1)/3} \phi^{2p} + \int |f|^{(4p+2)/7}\phi^{2p}  + 1 \right)
\end{eqnarray*}
so that for any $\delta < \displaystyle{\frac{1}{C_2}}$ we obtain
\begin{eqnarray*}
\int |\nabla_\epsilon z|^{p+1/2}\phi^{2p} & \leq & C_3 \left(\int |z|^{4p+2}\phi^{2p} + \int |\nabla_\epsilon^2 u|^{(4p+2)/3}\phi^{2p} \right.\\
&& \left. + \int |\nabla_\epsilon Y u|^{(2p+1)/3} \phi^{2p} + \int |f|^{(4p+2)/7}\phi^{2p}  + 1 \right)\ .
\end{eqnarray*}
By inserting this last expression into Lemma \ref{LEMcaccioppolisecond} we have the estimate
\begin{eqnarray*}
\int && \!\!\!\!\!\!\!\!\!\! |\nabla_\epsilon z|^{p+1/2}\phi^{2p} + \int|\nabla_\epsilon(|\nabla_\epsilon z|^{(p-1)/2})|^2\phi^{2p} \ \leq \ C_3 \left( \int|z|^{4p+2}\phi^{2p}\right. \\
&& \left. + \int |\nabla_\epsilon^2 u|^{(4p+2)/3}\phi^{2p} + \int |\nabla_\epsilon Yu|^{(2p+1)/3}\phi^{2p} + \int |f|^{(4p+2)/7}\phi^{2p} + 1 \right)\ .
\end{eqnarray*}
The statement of the theorem now follows by choosing a sequence of cutoff functions that converges to the characteristic function of the set $\Omega$.
\end{proof}

Iterating the previous theorem once, we obtain the following
estimates.
\begin{theo}
Let $u$ be a function satisfying (\ref{assumption}). Let $z$ be a smooth solution of equation (\ref{divergence}) and $f$ be a given $\C^{\infty}(\Omega)$ function, then for any $p \geq 3$ there exists a positive constant $C$ depending only on $p,\Omega$ and the constant $M$ of (\ref{assumption}) but independent of $\epsilon$ such that it holds
\begin{eqnarray*}
\|z\|_{W_\epsilon^{2,p+1/2}(\Omega)}^{p+1/2} \!\! & + & \!\! \| |\nabla_\epsilon^2 z|^{(p-1)/2} \|_{W_\epsilon^{1,2}(\Omega)}^{2} \ \leq \ C_4\left(\|z\|_{W_\epsilon^{1,4p+2}(\Omega)}^{4p+2} + \|u\|_{W_\epsilon^{2,(4p+2)/3}(\Omega)}^{(4p+2)/3}\right.\\
&& \!\!\!\!\!\!\!\!\!\!\! + \left. \|Yu\|_{W_\epsilon^{1,(2p+1)/3}(\Omega)}^{(2p+1)/3} + \|f\|_{W_\epsilon^{1,(4p+2)/7}(\Omega)}^{(4p+2)/7} + \|u\|_{W_\epsilon^{3,(2p+1)/3}(\Omega)}^{(2p+1)/3} + 1\right)\ .
\end{eqnarray*}
\end{theo}
\begin{proof}
If we call $s = X_k z$ then by Lemma \ref{LEMsecondderivatives} $M_{\epsilon,u} s = \rho$, where $\rho$ is given by
$$
\rho = X_k f - X_i\left(X_k \A_{ij}\right)X_j z - X_i \A_{ij}\omega^l_{k,j}X_l z - \omega^l_{k,i}X_l \A_{ij} X_j z\ .
$$
From Theorem \ref{THEOCaccioppoli1} we have
\begin{eqnarray*}
\|s\|_{W_\epsilon^{1,p+1/2}(\Omega)}^{p+1/2} & + & \| |\nabla_\epsilon s|^{(p-1)/2} \|_{W_\epsilon^{1,2}(\Omega)}^{2} \ \leq \ C_3\left(\|s\|_{L^{4p+2}(\Omega)}^{4p+2} + \|u\|_{W_\epsilon^{2,(4p+2)/3}(\Omega)}^{(4p+2)/3}\right.\\
& + & \left. \|Yu\|_{W_\epsilon^{1,(2p+1)/3}(\Omega)}^{(2p+1)/3} + \|\rho\|_{L^{(4p+2)/7}(\Omega)}^{(4p+2)/7} + 1\right)
\end{eqnarray*}
where, using $\lambda = (4p+2)/7$ and denoting with $C$ a positive constant depending only on $p,M$ that may change from line to line
\begin{eqnarray*}
\|\rho\|_{L^{\lambda}(\Omega)}^{\lambda} & \leq & C \left( \|X_k f\|_{L^{\lambda}(\Omega)}^{\lambda} + \|X_i\left(X_k \A_{ij}\right)X_j z\|_{L^{\lambda}(\Omega)}^{\lambda} \right.\\
&& \left. + \|X_i \A_{ij}\omega^l_{k,j}X_l z\|_{L^{\lambda}(\Omega)}^{\lambda} + \|\omega^l_{k,i}X_l \A_{ij} X_j z\|_{L^{\lambda}(\Omega)}^{\lambda}\right)\\
& \leq & C \left( \|f\|_{W_\epsilon^{1,\lambda}(\Omega)}^{\lambda} + \| |\nabla_\epsilon^3 u| |\nabla_\epsilon z| \|_{L^{\lambda}(\Omega)}^{\lambda} + \| |\nabla_\epsilon^2 u| |\nabla_\epsilon^2 z| \|_{L^{\lambda}(\Omega)}^{\lambda}\right.\\
&& \left. \| |\nabla_\epsilon^2 u| |\nabla_\epsilon z| \|_{L^{\lambda}(\Omega)}^{\lambda} + \| |\nabla_\epsilon Y u| |\nabla_\epsilon z| \|_{L^{\lambda}(\Omega)}^{\lambda} + \| |\nabla_\epsilon^2 z|\|_{L^{\lambda}(\Omega)}^{\lambda}\right)\ .
\end{eqnarray*}
Again, by Young inequality we can estimate each of these terms by
\begin{eqnarray*}
\| |\nabla_\epsilon^3 u| |\nabla_\epsilon z| \|_{L^{(4p+2)/7}(\Omega)}^{(4p+2)/7} & \leq & \| |\nabla_\epsilon^3 u| \|_{L^{(2p+1)/3}(\Omega)}^{(2p+1)/3} + \| |\nabla_\epsilon z| \|_{L^{4p+2}(\Omega)}^{4p+2}\\
\| |\nabla_\epsilon^2 u| |\nabla_\epsilon^2 z| \|_{L^{(4p+2)/7}(\Omega)}^{(4p+2)/7} & \leq & \frac{1}{\delta}\| |\nabla_\epsilon^2 u| \|_{L^{(4p+2)/3}(\Omega)}^{(4p+2)/3} + \delta\| |\nabla_\epsilon^2 z| \|_{L^{p+1/2}(\Omega)}^{p+1/2}\\
\| |\nabla_\epsilon^2 u| |\nabla_\epsilon z| \|_{L^{(4p+2)/7}(\Omega)}^{(4p+2)/7} & \leq & \| |\nabla_\epsilon^2 u| \|_{L^{(4p+2)/3}(\Omega)}^{(4p+2)/3} + \| |\nabla_\epsilon z| \|_{L^{4p+2}(\Omega)}^{4p+2} + c\\
\| |\nabla_\epsilon Y u| |\nabla_\epsilon z| \|_{L^{(4p+2)/7}(\Omega)}^{(4p+2)/7} & \leq & \| |\nabla_\epsilon Y u| \|_{L^{(2p+1)/3}(\Omega)}^{(2p+1)/3} + \| |\nabla_\epsilon z| \|_{L^{4p+2}(\Omega)}^{4p+2}\\
\| |\nabla_\epsilon^2 z|\|_{L^{(4p+2)/7}(\Omega)}^{(4p+2)/7} & \leq & \delta\| |\nabla_\epsilon^2 z| \|_{L^{p+1/2}(\Omega)}^{p+1/2} + \frac{c}{\delta}
\end{eqnarray*}
that proves the statement for $\delta < \displaystyle{\frac{1}{C_3}}$
\end{proof}

Iterating this last result we can then obtain Theorem \ref{THEOiterations}.

\section{A priori estimates II: approximation with a H\"ormander-type operator}

In this section we provide a new estimate of Sobolev type for
solutions of the equation (\ref{elliptic}), using the fact that the
equation can represented in the nondivergence form
\begin{equation}\label{nondivergence}
N_{\epsilon,u} u = a_{ij}(\nabla_\epsilon u)X_{i,u}X_{j,u} u = 0
\end{equation}
where the coefficients $a_{ij}$ are given by (\ref{coefficients}).

We first linearize this equation, \ie we apply the operator
$N_{\epsilon,u}$ to a sufficiently regular arbitrary function $z$.
This leads to
\begin{equation}\label{linearized}
N_{\epsilon,u} z = a_{ij}(\nabla_\epsilon u)X_{i,u}X_{j,u} z = 0\ .
\end{equation}

The aim of this section is to obtain  some a priori bounds of
Sobolev norms for solutions $z$ to the linearized equation
(\ref{linearized}). In particular we will need to estimate $L^p$
norm of second order derivatives of the function $z$ in the
direction of the vector fields $X_1, X_2$, which do not satisfy a
H\"ormander type condition uniform in $\e$. Hence we will
approximate these vector fields with $\C^{\infty}$ vector fields
which, together with their commutators of order 2, do not span the
space. Since at the moment we are not interested to higher
regularity of the solution, we do not require that the new vector
fields approximate the given one up to the third derivative, and we
can require, at the opposite, that they span the space at step 3.
Hence they can satisfy a Sobolev type inequality, which will be used to
estimate second order derivatives of the solution of the given
operator.

\begin{theo}\label{THEOfreesing}
Let $z$ be a classical solution of equation $N_{\epsilon,u}z = 0$, then
\begin{itemize}
\item[i)] if there exists a constant $\alpha \in ]0,1[$ and a constant $p > 10/3$ such that for any compact $K \subset\subset \Omega$ there exists a positive constant $C$ such that
$$
\|u\|_{\C_\epsilon^{1,\alpha}(K)} + \|Y z\|_{L^p(K)} + \|Y z\|_{W_\epsilon^{1,2}(K)} + \|z\|_{W_\epsilon^{2,2}(K)} \leq C
$$
then for any compact set $K_1 \subset\subset K$ there exists a positive constant $C_1$ depending on $K$, $C$ and $M$ such that
$$
z \in W_\epsilon^{2,10/3}(K_1)\ ;
$$
\item[ii)] if, in addition, there exists a positive constant $C'$ such that
$$
\|Y X_k z\|_{L^4(K)} \leq C'
$$
and $\alpha \geq 1/4$, then for every $p > 1$
$$
z \in W_\epsilon^{2,p}(K_1)\ .
$$
\end{itemize}
\end{theo}
We have denoted by $\C_\epsilon^{1,\alpha}$ the class of functions whose $X_{1,u}$ and $X_{2,u}$ derivatives are H\"older continuous of order $\alpha$. They are in particular Euclidean H\"older continuous of order $(1+\alpha)$, but the H\"older constant is not independent on $\epsilon$.

\subsection{Lifting and freezing}
We introduce the new vector fields
\begin{equation}\label{lifted}
\tx = X_1 + s^2 Y \ \ , \ \ty = X_2 \ \ , \ \ts = \p_s
\end{equation}
such that
\begin{displaymath}
[\ts,\tx] = \frac{2s}{\epsilon} \ty, \quad [\ts,[\ts,\tx]] = \frac{2}{\epsilon}\ty\ .
\end{displaymath}
The related exponential coordinates based at $(x_0,0)$, which we will call $(\tilde{e}_1,\tilde{e}_2,\tilde{e}_3)$, are defined as
\begin{equation}\label{expcoords}
(x_1,x_2,s) = \exp_{(x_0,0)}\big(\tilde{e}_1\tx+\tilde{e}_2\ty+\tilde{e}_3\ts\big)
\end{equation}
that is to say, given a curve $\gamma(t)$ such that $\gamma(0) = (x_0,0)$, whose derivative is given by the above vector field, \ie $\dot{\gamma}(t) = \tilde{e}_1\tx(\gamma(t))+\tilde{e}_2\ty(\gamma(t))+\tilde{e}_3\ts(\gamma(t))$, then the exponential coordinates in (\ref{expcoords}) satisfy $(x_1,x_2,s) = \gamma(1)$. They can be obtained by the following system
$$
\left\{\begin{array}{rcl}
\dot{x}_1 & = & \tilde{e}_1 \sigma_1^1 + \big( \tilde{e}_1 s^2 + \epsilon\tilde{e}_2\big)\sigma_2^1\vspace{0.2cm}\\
\dot{x}_2 & = & \tilde{e}_1 \sigma_1^2 + \big( \tilde{e}_1 s^2 + \epsilon\tilde{e}_2\big)\sigma_2^2\vspace{0.2cm}\\
\dot{s} & = & \tilde{e_3}\ .
\end{array}\right.
$$
Integrating we have
$$
\left\{\begin{array}{rcl}
s & = & \tilde{e_3}\vspace{0.2cm}\\
(x-x_0)_i & = & \tilde{e}_1 I_1^i + \tilde{e}_1 s^2 J^i + \epsilon\tilde{e}_2 I_2^i
\end{array}\right.
$$
where
\begin{eqnarray*}
I_j^i & = & \int_0^1 dt \ \sigma_j^i(x(t),u(x(t)))\\
J^i & = & \int_0^1 dt \ t^2 \sigma_2^i(x(t),u(x(t)))
\end{eqnarray*}
and the system
$$
\left\{\begin{array}{rcl}
\tilde{e}_1 I_1^1 + \tilde{e}_1 s^2 J^1 + \epsilon \tilde{e}_2 I_2^1 & = & (x-x_0)_1\vspace{0.2cm}\\
\tilde{e}_1 I_1^2 + \tilde{e}_1 s^2 J^2 + \epsilon \tilde{e}_2 I_2^2 & = & (x-x_0)_2
\end{array}\right.
$$
is solved implicitly by
\begin{equation}\label{expccord1}
\left\{\begin{array}{rcl}
\tilde{e}_1 & = & \displaystyle{\frac{(x-x_0)_1I_2^2 - (x-x_0)_2I_2^1}{(I_1^1 I_2^2 - I_1^2 I_2^1) + s^2 (J^1I_2^2 - J^2I_2^1)}}\vspace{0.2cm}\\
\epsilon\tilde{e}_2 & = & \displaystyle{\frac{(x-x_0)_2(I_1^1 + s^2 J^1) - (x-x_0)_1(I_1^2 + \epsilon s^2 J^2)}{(I_1^1 I_2^2 - I_1^2 I_2^1) + s^2 (J^1I_2^2 - J^2I_2^1)}}
\end{array}\right.
\end{equation}
where rank condition (\ref{rank}) ensures the well posedness of this solution, since $I_1^1 I_2^2 \neq I_1^2 I_2^1$.

We now introduce a freezing for the vector fields ${\tilde X}_j$. With the short-hand notation
$$
\Sigma_j^k = \frac{\sigma_j^k(x_0)}{\sigma_1^1(x_0) \sigma_2^2(x_0) - \sigma_1^2(x_0) \sigma_2^1(x_0)}
$$
we start approximating the exponential coordinates (\ref{expccord1}) in the following way
$$
\left\{\begin{array}{rcl}
\eo_1 & = & \displaystyle{\Sigma_2^2(x-x_0)_1 - \Sigma_2^1(x-x_0)_2}\vspace{0.2cm}\\
\epsilon \eo_2 & = & \displaystyle{\Sigma_1^1(x-x_0)_2 - \Sigma_1^2(x-x_0)_1 }
\end{array}\right.
$$
and define an analogue of the first order Taylor polynomial of a function $h:\Omega \rightarrow \R$
\begin{eqnarray}\label{1sttaylor}
P_{x_0}h (x) & = & h(x_0) + \eo_1(x)\tx h(x_0,0) + \eo_2(x) \ty h (x_0,0)\nonumber\\
& = & h(x_0) + \eo_1(x)X_1 h(x_0) + \eo_2(x) X_2 h (x_0)\ .
\end{eqnarray}
The freezed vector fields then read
\begin{equation}\label{liftfreezedvf}
\Xo_1 = P_{x_0}\sigma_1^i(x)\p_{x_i} + s^2 \Yo \ \ , \ \Xo_2 = \epsilon P_{x_0}\sigma_2^i(x)\p_{x_i} \ \ , \ \Xo_3 = \p_s
\end{equation}
where
$$
\Yo = P_{x_0}\sigma_2^i(x)\p_{x_i}\ .
$$
These are H\"ormander vector fields defining a step three stratified Lie algebra, whose first layer is spanned by $\{\Xo_1,\Xo_3\}$. We will denote by $\de$ the corresponding Carnot-Carath\'eodory distance, by $\de_{\epsilon}$ the full Riemannian distance and by $\Bo_\epsilon$ the balls relative to $\de_\epsilon$. Notice also that the homogeneous dimension of the space $(\R^3,\de)$ is 5.

\begin{rem}
It is well known that sub-Riemannian structures can be obtained with a limiting procedure from Riemannian counterparts \cite{Bellaiche}. More precisely, in this case we have that $(\R^3,\de_\epsilon)$ converge as a metric space in the Hausdorff-Gromov sense to $(\R^3,\de)$ as $\epsilon$ goes to zero. In particular for any $\xi, \eta \in \Omega \times (-1.1)$ we have that $\de_\epsilon(\xi,\eta) \to \de(\xi,\eta)$, that implies that given a threshold $\epsilon_0$ sufficiently small, there exists a positive constant $C=C(\epsilon_0)$ depending only on it such that for all $0 < \epsilon < \epsilon_0$ it holds
$$
|\Bo_\epsilon\big((x_0,0),R\big)| \geq C R^5\ .
$$
This means that for small $\epsilon$ the volume growth for the Riemannian metric measure space can be considered the same as for the sub-Riemannian one.
\end{rem}

\begin{lem}\label{LEMd1alpha}
If $h:\Omega\rightarrow\R$ is a $\C_E^{1,\alpha}$, \ie $h$ belongs to the Euclidean H\"older class $(1,\alpha)$, then there exists an $\epsilon_0 > 0$, a neighborhood $U$ of $x_0$ and a positive constant $C$ depending on $\epsilon_0$, $U$ but not on $\epsilon$ such that
\begin{equation}\label{d1alpha}
|h(x) - P_{x_0}h(x)| \leq C \de_\epsilon^{1+\alpha}(x,x_0)
\end{equation}
for all $x$ in $U$ and all $\epsilon < \epsilon_0$.
\end{lem}
\begin{proof}

We will first prove that
\begin{equation}\label{claimd1alpha}
|h(x) - P_{x_0}h(x)| \leq C d_E^{1+\alpha}(x,x_0)
\end{equation}
where $d_E$ stands for the Euclidean distance. Indeed this is sufficient for having (\ref{d1alpha}), since by the hypothesis of $u$ being $\C_E^{1,\alpha}$, in a sufficiently small neighborhood of $x_0$ the Euclidean distance is controlled by the Riemannian one with constant depending on the coefficients that define the vector fields (\ref{liftfreezedvf}), \ie the derivatives of $u$. Moreover, since $\Xo_1$ contains the unscaled direction $\Yo$, once $\epsilon$ is fixed under a given threshold $\epsilon_0$ it does not interfere anymore on the constant.

Let us first calculate exponential coordinates $(e_1,e_2,e_3)$ with respect to the vector fields (\ref{liftfreezedvf}), in a neighborhood of $(x_0,0)$. They are then given by
\begin{equation}\label{expcoord2}
(x_1,x_2,s) = \exp_{(x_0,0)}\big(e_1\Xo_1 + e_2\Xo_2 + e_3\Xo_3\big)
\end{equation}
that is
$$
\left\{\begin{array}{rcl}
e_1 & = & \displaystyle{\frac{(x-x_0)_1\Io_2^2 - (x-x_0)_2\Io_2^1}{(\Io_1^1 \Io_2^2 - \Io_1^2 \Io_2^1) + s^2 (\Jo^1\Io_2^2 - \Jo^2\Io_2^1)}}\vspace{0.2cm}\\
\epsilon e_2 & = & \displaystyle{\frac{(x-x_0)_2(\Io_1^1 + s^2 \Jo^1) - (x-x_0)_1(\Io_1^2 + s^2 \Jo^2)}{(\Io_1^1 \Io_2^2 - \Io_1^2 \Io_2^1) + s^2 (\Jo^1\Io_2^2 - \Jo^2\Io_2^1)}}
\end{array}\right.
$$
where
\begin{eqnarray*}
\Io_j^k & = & \int_0^1 dt \ P_{x_0}\sigma_j^k(x(t),u(x(t)))\\
& = & \sigma_j^k(x_0,u(x_0)) + X_i \sigma_j^k(x_0,u(x_0)) \int_0^1 dt \ \eo_i(x(t))\\
\Jo^i & = & \int_0^1 dt \ t^2 P_{x_0}\sigma_2^i(x(t),u(x(t)))\ .
\end{eqnarray*}

We now define in $\Omega$ an integral curve of (\ref{liftfreezedvf}), connecting $x_0$ with $x$
$$
\gamma(t) = \exp_{x_0}\left(t\big(e_1(x)\Xo_1 + e_2(x)\Xo_2\big)\right)
$$
so that, by Mean Value Theorem, for all $t \in (0,1)$ there exists a $t^* \in (0,t)$ such that
$$
h(\gamma(t)) - h(x_0) = e_1(x)\Xo_1h(\gamma(t^*)) + e_2(x)\Xo_2h(\gamma(t^*))\ .
$$
Thus, for all $t \in (0,1)$ we have
\begin{eqnarray*}
h(\gamma(t)) - P_{x_0}h(\gamma(t)) \!\!\!\!\!& = &\!\!\!\!\! h(\gamma(t)) - h(x_0) - \eo_1(\gamma(t))\tx h(x_0,0) - \eo_2(\gamma(t)) \ty h (x_0,0)\\
\!\!\!\!\!& = &\!\!\!\!\! e_i(x)\Xo_i h(\gamma(t^*)) - \eo_i(\gamma(t))\tilde{X}_i h(x_0,0)\\
\!\!\!\!\!& = &\!\!\!\!\! e_i(x) \big[ \Xo_i h(\gamma(t^*)) - \Xo_i h(x_0) \big] + \ e_i(x) \big[ \Xo_i h(x_0)\\
&&\!\!\!\! - \tilde{X}_i h(x_0,0) \big] + \ \big[ e_i(x) - \eo_i(\gamma(t)) \big] \tilde{X}_i h(x_0,0)
\end{eqnarray*}
where in particular the second term is zero since $\Xo_i h(x_0) = \tilde{X}_i h(x_0,0)$. For the first term, then, we observe that
\begin{eqnarray*}
|\Xo_1 h(\gamma(t^*)) - \Xo_1 h(x_0)| & \leq & C d_E^\alpha(x,x_0)\\
|\Xo_2 h(\gamma(t^*)) - \Xo_2 h(x_0)| & \leq & \epsilon C d_E^\alpha(x,x_0)
\end{eqnarray*}
because of the hypothesis of $h$ being $\C^{1,\alpha}$ and the smoothness of the coefficients, so that the constant $C$ depends also on the neighborhood of $x_0$ chosen. Then
$$
|e_i(x)| \ \big| \Xo_i h(\gamma(t^*)) - \Xo_i h(x_0) \big| \leq C d_E^{1+\alpha}(x,x_0)\ .
$$
For the third term, after introducing the notation
$$
\IIo_j^k = \frac{\Io_j^k}{\Io_1^1 \Io_2^2 - \Io_1^2 \Io_2^1}
$$
we have
$$
\left\{\begin{array}{rcl}
e_1(x,0) & = & \displaystyle{\IIo_2^2(x-x_0)_1 - \IIo_2^1(x-x_0)_2}\vspace{0.2cm}\\
\epsilon e_2(x,0) & = & \displaystyle{\IIo_1^1(x-x_0)_2 - \IIo_1^2(x-x_0)_1 }
\end{array}\right.
$$
so that
\begin{eqnarray*}
|e_1(x,0) - \eo_1(\gamma(t))| & = & \left| \IIo_2^2(x-x_0)_1 - \IIo_2^1(x-x_0)_2\right.\\
&& \left. \ - \ \Sigma_2^2(\gamma(t)-x_0)_1 + \Sigma_2^1(\gamma(t)-x_0)_2 \right|\\
& \leq & \left| \IIo_2^2(x-x_0)_1 - \Sigma_2^2(\gamma(t)-x_0)_1 \right|\\
&& \ + \ \left| \IIo_2^1(x-x_0)_2 - \Sigma_2^1(\gamma(t)-x_0)_2 \right|\\
\epsilon |e_2(x,0) - \eo_2(\gamma(t))| & = & \left| \IIo_1^1(x-x_0)_2 - \IIo_1^2(x-x_0)_1  \right.\\
&& \left. \ - \ \Sigma_1^1(\gamma(t)-x_0)_2 - \Sigma_1^2(\gamma(t)-x_0)_1 \right|\\
& \leq & \left| \IIo_1^1(x-x_0)_2 - \Sigma_1^1(\gamma(t)-x_0)_2 \right|\\
&& \ + \ \left| \IIo_1^2(x-x_0)_1 - \Sigma_1^2(\gamma(t)-x_0)_1 \right|\ .
\end{eqnarray*}
The generic term to control is then
$$
\Delta_{ij}^{\phantom{i}k} = \left| (x-x_0)_i \IIo_j^k - (\gamma(t)-x_0)_i \Sigma_j^k \right|\ .
$$
Explicitly, setting $D_1 = \Io_1^1 \Io_2^2 - \Io_1^2 \Io_2^1$ and $D_2 = \sigma_1^1(x_0) \sigma_2^2(x_0) - \sigma_1^2(x_0) \sigma_2^1(x_0)$
\begin{eqnarray*}
\Delta_{ij}^{\phantom{i}k} & = & \left| (x-x_0)_i \frac{\int_0^1 d\tau \ P_{x_0}\sigma_j^k(\gamma(\tau))}{D_1} - (\gamma(t)-x_0)_i \frac{\sigma_j^k(x_0)}{D_2}\right|\\
& \leq & \big|(x-x_0)_i\big|\left|\frac{\int_0^1 d\tau \ \big(P_{x_0}\sigma_j^k(\gamma(\tau)) - \sigma_j^k(\gamma(\tau))\big)}{D_1}\right|\\
&& + \ \left|(x-x_0)_i\int_0^1d\tau\sigma_j^k(\gamma(\tau))\right|\left|\frac{D_1 - D_2}{D_1D_2}\right|\\
&& + \ \left|\frac{(x-x_0)_i\int_0^1d\tau\sigma_j^k(\gamma(\tau)) - (\gamma(t)-x_0)_i\sigma_j^k(x_0)}{D_2}\right|\\
& = & \frac{N_1}{D_1} + \frac{N_{12}}{D_1D_2} + \frac{N_2}{D_2} \leq C (N_1 + N_{12} + N_2)\ .
\end{eqnarray*}
Where the last transition holds since rank condition (\ref{rank}) ensures that the $D_i$'s do not vanish, so that they can be bounded from below uniformly on compacts, and $C$ is a constant depending on the compact set. Moreover
$$
N_{12} \leq C d_E^{2}(x,x_0)
$$
indeed
\begin{eqnarray*}
N_{12} & \leq & \left|(x-x_0)_i\int_0^1 d\tau\sigma_j^k(\gamma(\tau))\right| \left( \left|X_i\left(\sigma_1^1 \sigma_2^2 - \sigma_1^2 \sigma_2^1\right)(x_0)\int_0^1d\tau \eo_i(\gamma(\tau))\right|\right.\\
&&\!\!\!\!\!\!\!\!\! + \left. \left|(X_i\sigma_1^1) (X_j\sigma_2^2) \int_0^1d\tau \eo_i \int_0^1d\tau \eo_j - (X_i\sigma_1^2) (X_j\sigma_2^1) \int_0^1d\tau \eo_i \int_0^1d\tau \eo_j\right|\right)
\end{eqnarray*}
where
$$
\left|(x-x_0)_i\int_0^1 d\tau\sigma_j^k(\gamma(\tau))\right| \leq C d_E(x,x_0)
$$
since the functions $\sigma_j^k$ are smooth, so they and their derivatives are locally bounded, and
$$
\left|\int_0^1d\tau \eo_1(\gamma(\tau))\right| \leq \int_0^1d\tau \left|\Sigma_2^2(\gamma(\tau)-x_0)_1 - \Sigma_2^1(\gamma(\tau)-x_0)_2\right| \leq C d_E(x,x_0)
$$
and the same holds for $\eo_2$, where the $1/\epsilon$ in the integrand is compensated by the $\epsilon$ in the vector field.

Writing for $N_1$ and $N_2$
\begin{eqnarray*}
N_1 & \leq & d_E(x,x_0) \int_0^1 d\tau \ \left|\sigma_j^k(\gamma(\tau)) - P_{x_0}\sigma_j^k(\gamma(\tau))\right|\\
N_2 & \leq & \int_0^1 d\tau \left|(x-x_0)_i\sigma_j^k(\gamma(\tau)) - (\gamma(t)-x_0)_i\sigma_j^k(x_0)\right|
\end{eqnarray*}
we obtain
\begin{eqnarray}\label{freezing1}
|h(\gamma(t)) \!\!\!\! & - & \!\!\!\!P_{x_0}h(\gamma(t))| \ \leq \ C \bigg(d_E^{1+\alpha}(x,x_0)\\
& + & \sum_{ijk}\int_0^1 d\tau \left|(x-x_0)_i\sigma_j^k(\gamma(\tau)) - (\gamma(t)-x_0)_i\sigma_j^k(x_0)\right| \nonumber\\
& + & d_E(x,x_0)\sum_{jk} \int_0^1 d\tau \ \left|\sigma_j^k(\gamma(\tau)) - P_{x_0}\sigma_j^k(\gamma(\tau))\right|\bigg)\ . \nonumber
\end{eqnarray}
Now if we set $t = 1$, inequality (\ref{freezing1}) reads
\begin{eqnarray*}
|h(x) \!\!\!\! & - & \!\!\!\!P_{x_0}h(x)| \ \leq \ C \bigg(d_E^{1+\alpha}(x,x_0)\\
& + & d_E(x,x_0)\sum_{jk}\int_0^1 d\tau \left|\sigma_j^k(\gamma(\tau)) - \sigma_j^k(x_0)\right| \\
& + & d_E(x,x_0)\sum_{jk}\int_0^1 d\tau \ \left|\sigma_j^k(\gamma(\tau)) - P_{x_0}\sigma_j^k(\gamma(\tau))\right|\bigg)\\
& \leq & C \bigg(d_E^{1+\alpha}(x,x_0) + d_E(x,x_0)\sum_{jk}\int_0^1 d\tau \ \left|\sigma_j^k(\gamma(\tau)) - P_{x_0}\sigma_j^k(\gamma(\tau))\right|\bigg)
\end{eqnarray*}
so that in order to complete the proof we need only to prove the following claim
$$
\left|\sigma_j^k(\gamma(\tau)) - P_{x_0}\sigma_j^k(\gamma(\tau))\right| \leq d_E(x,x_0)\ .
$$
To do so, we make use of inequality (\ref{freezing1}) in the case $h(x)=\sigma_\alpha^\beta(x)$, where for $x$ sufficiently close to $x_0$ the last term at the right hand side can be reabsorbed, so the claim follows by noting that
\begin{eqnarray*}
\left|(x-x_0)_i\sigma_j^k(\gamma(\tau)) - (\gamma(t)-x_0)_i\sigma_j^k(x_0)\right| & \leq & \left|(x-x_0)_i\right| \left|\sigma_j^k(\gamma(\tau)) - \sigma_j^k(x_0)\right|\\
&& + \left|(x-\gamma(t))_i\right| \left|\sigma_j^k(x_0)\right|\\
& \leq & C d_E(x,x_0)\ .
\end{eqnarray*}
This proves (\ref{claimd1alpha}) and hence the full statement.
\end{proof}

\begin{rem}\label{REMdiff}
The difference between the frozen vector fields and the original ones are expressed by
\begin{eqnarray*}
X_1 - \Xo_1 & = & \left(\sigma_1^i(x) - P_{x_0}\sigma_1^i(x) - s^2\sigma_2^i(x)\right)\p_{x_i}\\
X_2 - \Xo_2 & = & \epsilon\left(\sigma_2^i(x) - P_{x_0}\sigma_2^i(x)\right)\p_{x_i}\ .
\end{eqnarray*}
In view of Lemma \ref{LEMd1alpha}, and because of the choice of lifting up to step three, we get the estimate
\begin{eqnarray}\label{difference}
|X_i z - \Xo_i z| & \leq & \de_\epsilon^{1+\alpha}((x,s),(x_0,0))|\nabla_E z|\nonumber\\
& \leq & \de_\epsilon^{1+\alpha}((x,s),(x_0,0))(|\nabla_\epsilon z| + |Y z|)\ .
\end{eqnarray}
Indeed if we had $s$ instead of $s^2$ that would not hold true. This will be a key step for concluding the proof of Theorem \ref{THEOfreesing}.
\end{rem}

\subsection{The fundamental solution of the frozen operator}
If we extend to three dimensions the $2\times 2$ matrix $a_{ij}(\nabla_\epsilon u(x_0))$ to a $3 \times 3$ matrix keeping the same form (\ref{coefficients}) but with respect to the lifted vector fields (\ref{lifted}), \ie we define
$$
\ao_{ij} = \ao_{ij}(\tilde\nabla u(x_0,0)) = \delta_{ij} - \frac{\tilde{X_i}u(x_0,0)\tilde{X_j}u (x_0,0)}{1+|\tilde\nabla u(x_0,0)|^2}
$$
then with the new vector fields (\ref{liftfreezedvf}) we can define a uniformly subelliptic operator acting on a smooth function $z$ generally defined on $\Omega\times(-1,1)$
\begin{equation}\label{No}
\No z \doteq \sum_{i,j \in \{1,3\}} \ao_{ij}\Xo_i\Xo_j z
\end{equation}
that, due to the uniform ellipticity of the matrix $\ao$, behaves like a sublaplacian with respect to the corresponding stratified Lie algebra. Its elliptic regularization contains also the vector field $\Xo_2$ and reads
\begin{equation}\label{Noe}
\No_\epsilon z \doteq \sum_{i,j = 1}^3 \ao_{ij}\Xo_i\Xo_j z\ .
\end{equation}
Operator (\ref{Noe}) will be used in this section to approximate the original operator in nondivergence form (\ref{nondivergence}). The main advantage in this approximation consists of the fact that the limit operator for (\ref{Noe}) is subelliptic, while the limit operator for (\ref{nondivergence}) is not.

We will make use of the fundamental solution to (\ref{Noe}), which will be denoted by $\Go_\epsilon(\xi|\eta)$. Its variables $\xi,\eta$ are then in $\Omega\times(-1,1)$. For the sake of simplicity we will also make use of the following notation
\begin{eqnarray*}
(\Go_\epsilon \ast f) (\xi) = \int \Go_\epsilon(\xi|\eta) f(\eta)d\eta
\end{eqnarray*}
and when a derivation is applied to $\Go_\epsilon(\xi|\eta)$, we will mean it is made with respect to its first variable.

The following estimates on the fundamental solution are proved in \cite{CMa}.
\begin{theo}\label{TEOCittiManfredini}
Let $\Go_\epsilon(\xi|\eta)$ be the fundamental solution of (\ref{Noe}) and $x_0 \in \Omega$. Then for all compact sets $K \subset \subset \Omega \times (-1,1)$, for all $k \in \N$ and every multiindex $I = (i_1, \dots , i_k)$ with $i_j \in \{1,2,3\}$ there exists a positive constant $C$ independent of $\epsilon$ such that
\begin{displaymath}
|\no_\epsilon^I \Go_\epsilon(\xi|\eta)| \leq C \frac{\de_\epsilon^{2-k}(\xi,\eta)}{|\Bo_\epsilon(\xi,\de_\epsilon(\xi,\eta))|}
\end{displaymath}
\end{theo}

\begin{prop}\label{PROPYoung}
Let $f$ be in $L^q(\Omega)$ and extend it to a function on $\Omega \times (-1,1)$ by setting $f(x,s) \doteq f(x)$. Let $\mathcal{K}:[\Omega \times (-1,1)]^2 \to \R$ be a kernel satisfying
\begin{equation}\label{kernel}
|\mathcal{K}(\xi,\eta)| \leq C \frac{\de_\epsilon^{\kappa}(\xi,\eta)}{|\Bo_\epsilon(\xi,\de_\epsilon(\xi,\eta))|}\ .
\end{equation}
Then a function $u(\xi)$ defined as
$$
u(\xi) = \int \mathcal{K}(\xi,\eta)f(\eta)d\eta
$$
and such that $u(x,s) = u(x)$ satisfies
$$
\|u\|_{L^r(\Omega)} \leq C \|f\|_{L^q(\Omega)}
$$
with $5-\kappa q > 0$ and $r = \frac{5q}{5-\kappa q}$.
\end{prop}
The idea is that, by the homogeneity of the measure on stratified groups, the requirement (\ref{kernel}) implies that $\mathcal{K}$ is in $L^s$ with $s < \frac{Q}{Q-\kappa}$ where $Q$ is the homogeneous dimension of the group, in this case $Q = 5$. An argument like the standard Young inequality for convolutions implies then $\|u\|_r \leq C \|\mathcal{K}\|_s\|f\|_q$ with $1/s + 1/q = 1 + 1/r$, so $r < \frac{Qq}{Q - \kappa q}$.

For the sake of clarity, we reproduce here the main arguments of the proof, addressing the reader to \cite{CiPaPo} and \cite{CCM1} for the full detailed computation of representation formulae.

\begin{proof}[Proof of Theorem \ref{THEOfreesing}]
Let $z$ be a classical solution to $N_{\epsilon,u}z = 0$, hence if considered as a function on $\Omega\times(-1,1)$ it will not have dependence on the $s$ variable, \ie $z(\xi) = z(x_0,0)$.
Using the convention of summation over repeated indices, by definition of fundamental solution we have
\begin{eqnarray*}
z\phi & = & \Go_\epsilon \ast \No_\epsilon z \phi\\
& = & \Go_\epsilon \ast \left( z\No_\epsilon \phi + 2\ao_{ij}(\Xo_iz)(\Xo_j\phi)\right) + \Go_\epsilon \ast \left(\No_\epsilon z - N_\epsilon z\right) \phi
\end{eqnarray*}
where the term involving the difference between the freezed and the original operator, since $z$ does not depend on $s$, is
\begin{eqnarray*}
\No_\epsilon z - N_\epsilon z & = & \ao_{ij}\Xo_i\Xo_j z - a_{ij}(\nabla_\epsilon u)X_iX_j z\\
& = & (\ao_{ij}-a_{ij}(\nabla_\epsilon u)) X_iX_j z + \ao_{ij}(\Xo_i\Xo_j z - X_iX_j z)
\end{eqnarray*}
and in particular the last term can be written as
\begin{displaymath}
\Xo_i\Xo_j z - X_iX_j z = (\Xo_i - X_i) X_j z + \Xo_i (\Xo_j - X_j) z\ .
\end{displaymath}
We then obtain the representation formula
\begin{eqnarray}\label{repr}
z\phi & = & \Go_\epsilon \ast \Big( z\No_\epsilon \phi + 2\ao_{ij}(\Xo_iz)(\Xo_j\phi)\Big) + \Go_\epsilon \ast  \Big(\phi(\ao_{ij}-a_{ij}(\nabla_\epsilon u)) X_iX_j z\Big)\nonumber\\
&& + \Go_\epsilon \ast \Big(\phi\ao_{ij}(\Xo_i - X_i) X_j z\Big) + \Go_\epsilon \ast \Big(\ao_{ij}(\Xo_i \phi) (\Xo_j - X_j)z\Big)\nonumber\\
&& + \Xo_i \Go_\epsilon \ast \Big(\ao_{ij} \phi (\Xo_j - X_j) z \Big)\ .
\end{eqnarray}
To get a $W^{2,p}$ estimate of $z$ we perform second derivatives on the above representation formula. Arguing as in \cite{CCM1}, we can represent frozen second derivatives at the freezing point $x_0$, \ie $\Xo_k\Xo_l z\phi (x_0,0)$, simply by putting the derivatives on $\Go_\epsilon$, hence obtaining a representation formula similar to the above one but with $\Xo_k\Xo_l \Go ((x_0,0),\eta)$. Proceeding in this direction, the worst term to estimate is then the one corresponding to the last term in (\ref{repr}), since it contains third derivatives of $\Go_\epsilon$. Making use of Remark (\ref{REMdiff}), it becomes
\begin{eqnarray*}
\int \!\!\!\!\!\! && \!\!\!\!\!\!\left| \Xo_k \Xo_l \Xo_i \Go_\epsilon ((x_0,0)|\eta) \Big(\ao_{ij} \phi (\Xo_j - X_j) z (\eta) \Big)\right| d\eta\\
& \leq & C \int \mathcal{K}((x_0,0),\eta)(|\nabla_\epsilon z(\eta)| + |Y z(\eta)|)  \phi d\eta\\
\end{eqnarray*}
where by (\ref{difference}) and Theorem \ref{TEOCittiManfredini}
$$
\mathcal{K} = \left| \Xo_k \Xo_l \Xo_i \Go_\epsilon ((x_0,0)|\eta)\right| \de_\epsilon^{1+\alpha}(x,x_0) \leq C \frac{\de_\epsilon^{\alpha}((x_0,0),\eta)}{|\Bo_\epsilon((x_0,0),\de_\epsilon((x_0,0),\eta))|}\ .
$$

Applying Proposition \ref{PROPYoung} to this term and handling the others in the same way, we get the boundedness of the function $x_0 \rightarrow |\Xo_k \Xo_l z(x_0)|$ in the $L^r$ norm, with $r = \frac{5q}{5 - \alpha q}$ where $q$ is by hypothesis bigger than $10/3$, so also $r > 10/3$.

\end{proof}

\section{Uniform Sobolev regularity for the nonlinear approximating PDE}

In this section we unify the results obtained up to now, coming back to the nonlinear regularized equation (\ref{elliptic}). We will prove that

\begin{theo}\label{THEOmainreg}
Let $u$ be a smooth solution to $L_{\epsilon,u}u=0$ in $\Omega \subset \R^2$, then for every open set $\Omega_1 \Subset \Omega$, every $p \geq 3$ and every $m \geq 2$ there exists a positive constant $C$ depending on $p$, $m$, $\Omega_1$ and the constant $M$ of (\ref{assumption}) but is independent of $\epsilon$ such that
\begin{displaymath}
\|u\|_{W_\epsilon^{m,p}(\Omega_1)} + \|Y u\|_{W_\epsilon^{m,p}(\Omega_1)} \leq C\ .
\end{displaymath}
\end{theo}

\subsection{First iteration step}

\subsubsection*{First Euclidean Caccioppoli inequality: $\C^{1,\alpha}$ regularity}
We start providing a regularity sufficient to apply Theorem \ref{THEOfreesing}. This can be obtained by an Euclidean Caccioppoli type estimate for solutions to (\ref{elliptic}), uniform in $\epsilon$, which allows to perform a classical Moser iteration.

\begin{lem}
Let $u$ be a smooth solution of $L_{\epsilon,u}u = 0$, then for
every $q \geq 2$ there exits a positive constant $C$ independent of
$\epsilon$ such that for any cutoff function $\phi$ it holds
$$
\int |\nabla_E (|\nabla_\epsilon u|^{q/2})|^2\phi^{2p} \leq C \int |\nabla_\epsilon u|^q\left(\phi^2 + |\nabla_E \phi|^2\right)\phi^{2p-2}\ .
$$
\end{lem}
\begin{proof}
We first prove the following claim:
\begin{equation}\label{claim}
\int |\nabla_\epsilon (|\nabla_\epsilon u|^{q/2})|^2\phi^{2p} \leq C \int |\nabla_\epsilon u|^q\left(\phi^2 + |\nabla_E \phi|^2\right)\phi^{2p-2}\ .
\end{equation}
In order to do this, we note that setting $z = X_k u$ we can rewrite the conclusion of Lemma \ref{LEMcaccioppolifirst} as
$$
\int |\nabla_\epsilon (|z|^{q/2})|^2\phi^{2p} \leq C \left( \int |z|^q\left(\phi^2 + |\nabla_\epsilon \phi|^2\right)\phi^{2p-2} + \int |f| |z|^{q-1}\phi^{2p}\right)
$$
by choosing the exponent $p$ of the function $z$ as equal to $q-1$ while leaving unvaried the exponent of the cutoff function (the proof of Lemma \ref{LEMcaccioppolifirst} shows that the two exponents need not to be the same). Let us now look at the last term. By Corollary \ref{CORfz} we have
$$
|f| \leq C(|\nabla_\epsilon u| + |\nabla_\epsilon z|)
$$
so that
\begin{eqnarray*}
\int |f| |z|^{q-1}\phi^{2p} & \leq & C \left(\int |\nabla_\epsilon u|^q\phi^{2p} + \int |\nabla_\epsilon z| |z|^{q-1}\phi^{2p}\right)\\
& \leq & C \left( \left[1 + \frac{1}{\delta}\right]\int |\nabla_\epsilon u|^{q}\phi^{2p} + \delta\int |\nabla_\epsilon z|^2 |z|^{q-2}\phi^{2p}\right)
\end{eqnarray*}
and
$$
|\nabla_\epsilon z|^2 |z|^{q-2} = \frac{4}{q^2}|\nabla_\epsilon (|z|^{q/2})|^2\ .
$$
Noting then that $|\nabla_\epsilon \phi| \leq |\nabla_E \phi|$, claim (\ref{claim}) is proved. The same argument holds for $|Y(|\nabla_\epsilon u|^{q/2})|^2$, which provides the full statement since for any given function $h$
\begin{equation}\label{nablaEeps}
|\nabla_E h| \leq C \left(|\nabla_\epsilon h| + |Y h|\right)\ .
\end{equation}
\end{proof}

The previous lemma allows us to obtain a statement analogous to that of \cite[Proposition 3.7]{CCMn}, \ie
\begin{theo}\label{THEOeuclidean}
Let $u$ be a solution of $L_{\epsilon,u}u = 0$ in $\Omega \subset \R^2$, then for every compact set $K \Subset \Omega$ there exists a positive constant $C$ such that
$$
\|u\|_{\C_\epsilon^{1,\alpha}(K)} + \|u\|_{W_\epsilon^{2,2}(K)} + \|Yu\|_{W_\epsilon^{1,2}(K)} \leq C
$$
where the class $\C_\epsilon^{1,\alpha}$ is the same as was defined for Theorem \ref{THEOfreesing}.
\end{theo}
Indeed the previous Caccioppoli estimates works precisely on the $\nabla_\epsilon$ derivatives of $u$, so $\C_\epsilon^{1,\alpha}$ is the natural class that can result from a Moser iteration.

We note that this is sufficient to apply Theorem \ref{THEOfreesing}, since when dealing with the nonlinear equation there is not anymore difference between ``solution'' $z$ and ``coefficients'' $u$, so in particular the $L^p_{loc}$ requirement for $Yz$ is implied by the $\C^{1,\alpha}$ boundedness of $u$.

\subsubsection*{Second Euclidean Caccioppoli inequality: higher integrability}

Proceeding analogously to the proof of Proposition \ref{PROPinterp1} and Lemma \ref{LEMcaccioppolisecond}, we obtain the following interpolation inequality and a priori Caccioppoli estimate.

\begin{prop}[Second interpolation inequality]\label{PROPinterp2}
For every $p \geq 3$ there exists a positive constant $C$ dependent
only on $p$ and the constant $M$ in (\ref{assumption}) such that for
every function $z \in \C^\infty(\Omega)$, every cutoff function
$\phi$  and every $\delta > 0$ it holds
\begin{eqnarray*}
\int |X_iz|^{p+1}\phi^{2p} & \leq & C \left[\int \left(z^{p+1}\phi^{2p} + z^2|X_iz|^{p-1}\phi^{2(p-1)}|X_i\phi|^2\right)\right.\\
&& + \left.\int |X_i^2z|^2|X_iz|^{p-3}z^2\phi^{2p}\right]\ .
\end{eqnarray*}
\end{prop}

\begin{cor}\label{CORestp+1}
Let $z$ be a smooth solution of equation (\ref{diveq1}) and let $f$
be locally summable in $\Omega$, then for all $p\geq 3$ there exists
a positive constant $C_2'$ depending only on $p$ and the constant
$M$ in (\ref{assumption}) such that for any cutoff function $\phi$
it holds
\begin{eqnarray*}
\int \left|\nabla_\epsilon \left(|\nabla_\epsilon z|^{(p-1)/2}\right)\right|^2\phi^{2p}
& \leq & C_2' \left(\delta\int |\nabla_\epsilon z|^{p+1}\phi^{2p} + \frac{1}{\delta}\int |\nabla_\epsilon^2 u|^{p+1}\phi^{2p} \right. \\
&& \!\!\!\!\!\!\left. + \,\frac{1}{\delta}\int |\nabla_\epsilon Yu|^{(p+1)/2} \phi^{2p} + \frac{1}{\delta}\int |f|^{(p+1)/2}\phi^{2p} + 1\right)\, .
\end{eqnarray*}
\end{cor}

These two results allow to obtain a second Euclidean Caccioppoli estimate for solutions to the nonlinear elliptic regularized equation, which combined with the Euclidean Sobolev inequality will provide the first step of the iteration.

\begin{lem}\label{LEMeuclideancaccioppoli2}
Let $u$ be a smooth solution of $L_{\epsilon, u}u = 0$. Then there exists a positive constant $C$ independent of $\epsilon$ such that its derivatives $z = X_ku$ satisfy
\begin{displaymath}
\int |\nabla_E\nabla_\epsilon z|^2\phi^6 \leq C \left( \int |\nabla_\epsilon z|^4\phi^6 + 1\right)\ .
\end{displaymath}
\end{lem}
\begin{proof}
Let us indicate $v = Yu$ and $z = X_k u$ and denote by $C$ a positive constant depending only on $p,M$ allowed to change from line to line.

By Corollary \ref{CORestp+1}, choosing $p=3$ and using Corollaries \ref{CORfz} and \ref{CORfv}, we get
\begin{eqnarray*}
\int |\nabla_\epsilon^2 z|^2\phi^6 & \leq & C \left(\delta\int |\nabla_\epsilon z|^4\phi^6 + \frac{1}{\delta}\left(\int |\nabla_\epsilon^2 u|^4 \phi^6 + \int |\nabla_\epsilon v|^2\phi^6 + 1\right)\right)\\
\int |\nabla_\epsilon^2 v|^2\phi^6 & \leq & C \left(\delta\int |\nabla_\epsilon v|^4\phi^6 + \frac{1}{\delta}\left(\int |\nabla_\epsilon^2 u|^4 \phi^6 + 1\right)\right)\ .
\end{eqnarray*}
Moreover, by Proposition \ref{PROPinterp2} with $p=3$ and using assumption (\ref{assCM}) we obtain
\begin{eqnarray*}
\int |\nabla_\epsilon z|^4 \phi^6 & \leq & C \left( \int |\nabla_\epsilon z|^2 \phi^4 |\nabla_\epsilon \phi|^2 + \int |\nabla_\epsilon^2 z|^2 \phi^6\right)\\
\int |\nabla_\epsilon v|^4 \phi^6 & \leq & C \left( \int |\nabla_\epsilon v|^2 \phi^4 |\nabla_\epsilon \phi|^2 + \int |\nabla_\epsilon^2 v|^2 \phi^6\right)\\
\end{eqnarray*}
where the first term is bounded by Theorem \ref{THEOeuclidean}, so that
\begin{eqnarray}
\int |\nabla_\epsilon z|^4\phi^6 & \leq & C\left( \int |\nabla_\epsilon^2 z|^2\phi^6 + 1 \right)\nonumber\\
\int |\nabla_\epsilon v|^4\phi^6 & \leq & C\left( \int |\nabla_\epsilon^2 v|^2\phi^6 + 1 \right)\label{intrv}
\end{eqnarray}
that provide the two intrinsic Caccioppoli inequalities
\begin{eqnarray}
\int |\nabla_\epsilon^2 z|^2\phi^6 & \leq & C \left(\int |\nabla_\epsilon^2 u|^4 \phi^6 + 1\right)\nonumber\\
\int |\nabla_\epsilon^2 v|^2\phi^6 & \leq & C \left(\int |\nabla_\epsilon^2 u|^4 \phi^6 + 1\right)\label{intrcaccv}\ .
\end{eqnarray}
Noting that
\begin{eqnarray*}
|\nabla_E X_j z| & \leq & |X_1 X_j z| + |Y X_j z| \leq |\nabla_\epsilon^2 z| + |X_j Y X_k u| + \frac{1}{\epsilon} |\omega^l X_l z|\\
& \leq & |\nabla_\epsilon^2 z| + |X_j X_k Y u| + \frac{1}{\epsilon}|X_j \omega^l X_l u| + \frac{1}{\epsilon} |\omega^l X_l z|\\
& \leq & |\nabla_\epsilon^2 z| + |\nabla_\epsilon^2 v| + \frac{1}{\epsilon} |(X_j\omega^l) X_l u| + \frac{1}{\epsilon} |\omega^l X_j X_l u| + \frac{1}{\epsilon} |\omega^l X_l z|
\end{eqnarray*}
where
\begin{eqnarray*}
\frac{1}{\epsilon} |(X_j\omega^l) X_l u| & = & \left| \left(X_j \frac{\omega^1}{\epsilon}\right)X_1 u + (X_j\omega^2)Yu \right| \leq C |\nabla_\epsilon^2 u|\\
\frac{1}{\epsilon} |\omega^l X_j X_l u| & = & \left|\frac{\omega^1}{\epsilon}X_jX_1 u + \omega^2 X_j Y u\right| \leq C \left( |\nabla_\epsilon^2 u| + |\nabla_\epsilon v|\right)\\
\frac{1}{\epsilon} |\omega^l X_l z| & = & \left| \frac{\omega^1}{\epsilon} X_1 z + \omega^2 Y z\right| \leq C\left(|\nabla_\epsilon z| + |\nabla_\epsilon v| + 1\right)
\end{eqnarray*}
so that
\begin{equation}\label{nablaE}
|\nabla_E X_j z| \leq C \left( |\nabla_\epsilon^2 z| + |\nabla_\epsilon^2 v| + |\nabla_\epsilon z| + |\nabla_\epsilon v| + |\nabla_\epsilon^2 u| + 1\right)
\end{equation}
we end up with the desired claim.
\end{proof}

\begin{prop}\label{PROPfirstiteration}
Let $u$ be a solution of $L_{\epsilon,u}u = 0$ in $\Omega \subset \R^2$, then for every $\Omega_1\Subset\Omega$ and every $p \geq 1$ there exists a positive constant $C$ independent of $\epsilon$ such that
\begin{itemize}
\item[i)] $\|u\|_{W_\epsilon^{2,p}(\Omega_1)} \leq C$
\item[ii)] $\|Yu\|_{W_\epsilon^{1,p}(\Omega_1)} \leq C$
\end{itemize}
\end{prop}
\begin{proof}
To prove i) we make use of Lemma \ref{LEMeuclideancaccioppoli2} together with the Euclidean Sobolev inequality in $\R^2$. Let us indicate $z = X_ku$ and denote with $\phi$ a cutoff function. Then we get, for any $2 \leq r < \infty$
\begin{eqnarray*}
\left(\int\left(|\nabla_\epsilon z|\phi^3\right)^r\right)^{2/r} & \leq & \int \left|\nabla_E \left(\nabla_\epsilon z \phi^3\right)\right|^2 \ \leq \ C \left[ \int |\nabla_\epsilon z|^4\phi^6 + 1\right]\\
& & \!\!\!\!\!\!\!\!\!\!\!\!\!\!\!\!\!\!\!\!\!\!\!\! \leq \ C\left[ \left(\int \left(|\nabla_\epsilon z|^\alpha\phi^{3\alpha}\right)^q\right)^{\frac{1}{q}}\left(\int \left(|\nabla_\epsilon z|^{4-\alpha}\phi^{6-3\alpha}\right)^{\frac{q}{q-1}}\right)^{\frac{q-1}{q}} + 1 \right]
\end{eqnarray*}
where the last transition is H\"older inequality. We now choose $\alpha$ and $q$ such that $\alpha q = r$ and $1/q = 2/r$, \ie $\alpha = 2$ and $q = r/2$ so that
\begin{equation}\label{tmpest}
\left(\int\left(|\nabla_\epsilon z|\phi^3\right)^r\right)^{2/r} \leq C \left[ \left(\int \left(|\nabla_\epsilon z|\phi^3\right)^r\right)^{\frac{2}{r}}\left(\int_{\textrm{supp}(\phi)} |\nabla_\epsilon z|^{\frac{2r}{r-2}}\right)^{\frac{r-2}{r}} + 1 \right]
\end{equation}
By Theorem \ref{THEOeuclidean} we can apply Theorem \ref{THEOfreesing} i), so that if we set $r=6$ we have $2r/(r-2) = 3$, and
$$
\int_{\textrm{supp}(\phi)} |\nabla_\epsilon z|^3 \leq \left(\int_{\textrm{supp}(\phi)} |\nabla_\epsilon z|^{\frac{10}{3}}\right)^{9/10} 
|\textrm{supp}(\phi)|^{1/10}
$$
can be chosen arbitrarily small with the support of $\phi$. We emphasize that for this point it was crucial to have an integrability up to a power higher than 3. This could be obtained through estimates which could be applied with the full strength of the H\"ormander setting by making use of the freezing technique.

Since estimate (\ref{tmpest}) holds for $k = 1,2$ we now obtain
$$
\int |\nabla_\epsilon^2 u|^6\phi^{18} \leq C\ .
$$
By inequalities (\ref{intrv}) and (\ref{intrcaccv}) we then have
\begin{eqnarray}\label{p4}
\int |\nabla_\epsilon v|^4\phi^6 & \leq & \int |\nabla_\epsilon^2 v|^2\phi^6 \leq C \left(\int |\nabla_\epsilon^2 u|^4 \phi^6 + 1\right)\\
& \leq & C \left(\int |\nabla_\epsilon^2 u|^6 \phi^{18} + \int_{\textrm{supp}(\phi)} |\nabla_\epsilon^2 u|^3 + 1\right) \leq C\ .\nonumber
\end{eqnarray}
Now since
$$
|\nabla_E \nabla_\epsilon u| \leq C \left(|\nabla_\epsilon^2 u| + |\nabla_\epsilon v| + 1\right)
$$
we have proved that $|\nabla_E \nabla_\epsilon u| \in L^4_{loc}$, and consequently by the Euclidean Sobolev-Morrey inequality in $\R^2$ we have that $\nabla_\epsilon u$ is Euclidean H\"older continuous of order $1/2$. This way we can apply the conclusion of Theorem \ref{THEOfreesing} ii) and obtain point i).

For point ii) we note that we can apply Theorem \ref{THEOCaccioppoli1} to $v=Yu$:
\begin{eqnarray*}
\|v\|_{W^{1,p+1/2}(\Omega)}^{p+1/2} & \leq & C_3\left(\|v\|_{L^{4p+2}(\Omega)}^{4p+2} + \|u\|_{W^{2,(4p+2)/3}(\Omega)}^{(4p+2)/3}\right.\\
&& \left. \ + \|v\|_{W^{1,(2p+1)/3}(\Omega)}^{(2p+1)/3} + \|f_v\|_{L^{(4p+2)/7}(\Omega)}^{(4p+2)/7} + 1\right)
\end{eqnarray*}
where at the right hand side the first term is bounded by assumption (\ref{assCM}), the second term is point i) we proved, and the third and fourth terms can be bounded using Corollary \ref{CORfv} by
$$
\|\nabla_\epsilon v\|_{L^{(2p+1)/3}(\Omega)}^{(2p+1)/3} + \|f_v\|_{L^{(4p+2)/7}(\Omega)}^{(4p+2)/7} \leq C\|\nabla_\epsilon v\|_{L^p(\Omega)}^p
$$
that is bounded for $p=4$ by (\ref{p4}) and then for all $p$.
\end{proof}

\subsection{Proof of the main iteration step}

\begin{lem}\label{LEMextra}

Let us assume (\ref{assumption}). Then there exists a positive constant $C$ depending only on $\Omega$ and $M$ of (\ref{assumption}) such that for every $\delta$ sufficiently small it holds
\begin{displaymath}
\int |\nabla_\epsilon z|^{7/3} \phi^{4} \leq C \left(\delta \int |\nabla_\epsilon^2 z|^2 \phi^{4} + \delta \int |\nabla_\epsilon z|^2\phi^{4} + \frac{1}{\delta} \int |z|^{3}\left( |\nabla_\epsilon \phi|^{3} + \phi^{3}  \right)\phi + 1\right)
\end{displaymath}
for every function $z \in \C^\infty(\Omega)$ and every cutoff
function $\phi$.
\end{lem}
\begin{proof}
The claim follows by choosing $p=2$ in the following interpolation inequality
\begin{eqnarray*}
\int |\nabla_\epsilon z|^{p+1/3} \phi^{2p} & \leq & C \left(\delta \int |\nabla_\epsilon (|\nabla_\epsilon z|^{p/2})|^2 \phi^{2p} + \delta \int |\nabla_\epsilon z|^p\phi^{2p}\right.\\
&& \left. + \ \frac{1}{\delta} \int |z|^{3p/2}\left( |\nabla_\epsilon \phi|^{3p/2} + \phi^{3p/2}  \right)\phi^{p/2} + 1\right)
\end{eqnarray*}
which can be proved in the same way as Proposition \ref{PROPinterp1}.
\end{proof}

\begin{proof}[Proof of Theorem \ref{THEOmainreg}]
Let us denote $z = X_k u$ and $v = Y u$. Then by Theorem \ref{THEOiterations} i) we have
\begin{eqnarray}\label{finale}
\|z\|_{W^{m,p+1/2}(\Omega)}^{p+1/2} + \|v\|_{W^{m,p+1/2}(\Omega)}^{p+1/2} & \leq & C_4\left(\|z\|_{W^{m-1,4p+2}(\Omega)}^{4p+2} + \|v\|_{W^{m-1,4p+2}(\Omega)}^{4p+2}\right.\nonumber\\
&& \!\!\!\!\!\!\!\!\!\!\!\!\!\!\!\!\!\!\!\!\!\!\!\!\!\!\!\left. + \ \|f_z\|_{W^{m-1,(4p+2)/7}(\Omega)}^{(4p+2)/7} + \|f_v\|_{W^{m-1,(4p+2)/7}(\Omega)}^{(4p+2)/7} \right.\nonumber\\
&& \!\!\!\!\!\!\!\!\!\!\!\!\!\!\!\!\!\!\!\!\!\!\!\!\!\!\!\left. + \ \|u\|_{W^{m,(4p+2)/3}(\Omega)}^{(4p+2)/3} + \|u\|_{W^{m+1,(2p+1)/3}(\Omega)}^{(2p+1)/3} + 1\right)\nonumber\\
& \!\!\!\!\!\!\!\!\!\!\!\!\!\!\!\!\!\!\!\!\!\!\!\!\!\!\!\!\!\!\!\leq & \!\!\!\!\!\!\!\!\!\!\!\!\!\!\!\!\!\!C\left(\|z\|_{W^{m-1,4p+2}(\Omega)}^{4p+2} + \|v\|_{W^{m-1,4p+2}(\Omega)}^{4p+2}\right.\nonumber\\
&&  \!\!\!\!\!\!\!\!\!\!\!\!\!\!\!\!\!\!\!\!\!\!\!\!\!\!\!\left. + \ \|z\|_{W^{m,(2p+1)/3}(\Omega)}^{(2p+1)/3} + \|v\|_{W^{m,(2p+1)/3}(\Omega)}^{(2p+1)/3} + 1\right)
\end{eqnarray}
where the last transition holds due to Corollaries \ref{CORfz} and \ref{CORfv} and H\"older inequality. Let us now rewrite the obtained inequality setting $q = p + 1/2$:
\begin{eqnarray*}
\|z\|_{W^{m,q}(\Omega)}^{q} + \|v\|_{W^{m,q}(\Omega)}^{q} & \leq & C\left(\|z\|_{W^{m-1,4q}(\Omega)}^{4q} + \|v\|_{W^{m-1,4q}(\Omega)}^{4q}\right.\\
&&  \left. + \ \|z\|_{W^{m,\frac{2}{3}q}(\Omega)}^{\frac{2}{3}q} + \|v\|_{W^{m,\frac{2}{3}q}(\Omega)}^{\frac{2}{3}q} + 1\right)
\end{eqnarray*}

By Proposition \ref{PROPfirstiteration} i) and ii) the first two terms at the right hand side are bounded for $m=2$, provided a restriction to a subsets $\Omega_1$ compactly contained in $\Omega$. We want to prove now, by induction, that for all $m \geq 2$
\begin{equation}\label{induction}
\|z\|_{W^{m,q}(\Omega_1)}^{q} + \|v\|_{W^{m,q}(\Omega_1)}^{q} \leq C \left(\|z\|_{W^{m,\frac{2}{3}q}(\Omega_1)}^{\frac{2}{3}q} + \|v\|_{W^{m,\frac{2}{3}q}(\Omega_1)}^{\frac{2}{3}q} + 1\right)\ .
\end{equation}

The induction scheme reads as follows: given that for any $m$ it holds
\begin{equation}\label{scheme1}
\|h\|_{m,q}^q \leq C \left(\|h\|_{m-1,4q}^{4q} + \|h\|_{m,\frac{2}{3}q}^{\frac{2}{3}q} + 1\right)
\end{equation}
and given the starting point $m=2$
$$
\|h\|_{2,q}^q \leq C \left(\|h\|_{2,\frac{2}{3}q}^{2/3q} + 1\right)
$$
we want to prove that if for a fixed $\overline{m}$ we have
\begin{equation}\label{scheme2}
\|h\|_{\overline{m},q}^q \leq C\left(\|h\|_{\overline{m},\frac{2}{3}q}^{\frac{2}{3}q} + 1\right)
\end{equation}
then it holds
$$
\|h\|_{\overline{m}+1,q}^q \leq C\left(\|h\|_{\overline{m}+1,\frac{2}{3}q}^{\frac{2}{3}q} + 1\right)\ .
$$
This is done using (\ref{scheme1})
\begin{displaymath}
\|h\|_{\overline{m}+1,q}^q \leq C\left(\|h\|_{\overline{m},4q}^{4q} + \|h\|_{\overline{m}+1,\frac{2}{3}q}^{\frac{2}{3}q} + 1\right)
\end{displaymath}
and noting that (\ref{scheme2}) implies that $\|h\|_{\overline{m},q}^q$ is bounded for all $q$ by bootstrapping, provided we can get a starting point. This is given by $q = 3$, using Theorem \ref{THEOiterations} ii). Indeed, repeating the argument used for inequality (\ref{finale}) with $p=3$ we get
\begin{eqnarray*}
\|z\|_{W^{m+1,2}(\Omega_1)}^{2} + \|v\|_{W^{m+1,2}(\Omega_1)}^{2}& \leq &  C\left(\|z\|_{W^{m-1,14}(\Omega_1)}^{14} + \|v\|_{W^{m-1,14}(\Omega_1)}^{14}\right.\\
&&  \left. + \ \|z\|_{W^{m,7/3}(\Omega_1)}^{7/3} + \|v\|_{W^{m,7/3}(\Omega_1)}^{7/3} + 1\right)
\end{eqnarray*}
where at the right hand side the first two terms are bounded for $m=2$ by Proposition \ref{PROPfirstiteration}, and the other two terms can be reabsorbed using the interpolation given by Lemma \ref{LEMextra}.
By (\ref{induction}), bootstrapping on $q$ we then obtain the desired result.
\end{proof}

\section{Intrinsic regularity for vanishing viscosity solutions}

This section is devoted to the proof of regularity for vanishing viscosity solutions introduced in Definition (\ref{defvanishingviscosity}), which provides a foliation result for intrinsic minimal graphs, \ie solutions to equation (\ref{equation}). Regularity estimates are obtained by making use of the estimates proved in the previous section, since they were uniform in $\epsilon$. In particular we will obtain H\"older continuity, that in definitive will be sufficient for $\C^\infty$ regularity, in the horizontal direction, while no regularity is expected in the transverse direction, in accordance with the geometry outlined by the foliation.

To come back to the notation of Definition (\ref{defvanishingviscosity}), we will indicate by $u$ a vanishing viscosity solution of (\ref{equation}) and by $(u_n)$ its approximating sequence referring to the vanishing positive real sequence $(\epsilon_n)$. The corresponding nonlinear vector fields will be denoted by $X_{1,n} = \sigma_1^j(x_1,x_2,u_n(x_1,x_2))\partial_{x_j}$, $X_{2,n} = \epsilon_n\sigma_2^j(x_1,x_2,u_n(x_1,x_2))\partial_{x_j}$ and $Y_n = \frac{1}{\epsilon_n} X_{2,n}$. Accordingly, we will use $\nabla_{\epsilon_n}$ and $W^{k,p}_{\epsilon_n}$ for the natural gradient and Sobolev spaces. For the limit equation we will use the notations $X = X_{1,\infty} = \sigma_1^j(x_1,x_2,u(x_1,x_2))\partial_{x_j}$, $\nabla_0 = (X,0)$, $W^{k,p}_0$, while $\nabla_E$ and $W^{k,p}_E$ still stand for the usual Euclidean gradient and Sobolev spaces.

\begin{theo}\label{theoXregularity}
Let $u$ be a vanishing viscosity solution of equation (\ref{equation}). Then
\begin{equation}\label{Xuregularity}
Xu \in W_{E,\textrm{loc}}^{1,p}(\Omega) \ \ \textrm{ for all } p > 1\ .
\end{equation}
\end{theo}
\begin{proof}
To prove (\ref{Xuregularity}) it suffices to obtain the following estimate uniform in $\epsilon$: for every ball $B \subset \subset \Omega$ and all $p > 1$ there exists a positive constant $C$ such that
\begin{equation}\label{W1p}
\|\nabla_{\epsilon_n}u_n\|_{W_E^{1,p}(B)}\leq C
\end{equation}
for all $n$, and this is a direct consequence of (\ref{nablaEeps}) and Theorem \ref{THEOmainreg} with $m = 2$.
\end{proof}

\begin{prop}\label{PROPweakequation}
Let $u$ be a vanishing viscosity solution of equation (\ref{equation}). Then
\begin{equation}\label{xj}
X_{1,n} u_n\rightarrow Xu, \quad X_{2,n} u_n \rightarrow 0
\end{equation}
as $n \rightarrow +\infty$ weakly in $W^{1,2}_{E, \loc}(\Omega)$. Moreover
equation (\ref{equation}) can be represented as $$X^2u =0$$ and is satisfied
 weakly in the Sobolev sense, and hence, pointwise a.e. in $ \Omega$, \ie
$$
\int_\Omega XuX^\dag\phi =0 \ \ \text{ for all }\phi\in \C^{\infty}_0(\Omega)\ .
$$
\end{prop}
\begin{proof}
Since the approximating solution $u_n$ is of class $\C^\infty$, we can use the nondivergence form of equation (\ref{equation})
$$
\sum_{i,j=1}^{2} a_{ij}(\nabla_{\e_n} u_n) X_{i,n}X_{j,n}u_n=0.
$$
Here
$$
a_{ij}(\nabla_{\e_n} u_n) \rightarrow a_{ij}(\nabla_0 u) = \delta_{i 1} \delta_{j1} \text{  in  } L^p\ ,
$$
while (\ref{W1p}) and the definition of vanishing viscosity solution imply
\begin{displaymath}
X_{1,n} u_n\rightarrow Xu, \quad X_{2,n} u_n  \rightarrow 0
\end{displaymath}
as $n \rightarrow +\infty$ weakly in $W^{1,2}_{E,loc}(\Omega)$.
Hence letting $n$ go to $\infty$ in the nondivergence form equation we conclude
$$
X^2u =0
$$
in the Sobolev sense.
\end{proof}

\begin{proof}[Proof of Theorem \ref{theoCregularity}]
By Theorem \ref{THEOmainreg} with $m = 3$ and inequality (\ref{nablaEeps}) we have that $X^2 u \in
W_{E,\textrm{loc}}^{1,p}(\Omega)$ for all $p > 1$. 
Then, by Sobolev-Morrey embedding in $\R^2$, $X^2 u$ is H\"older continuous, and
this implies that the equation $X^2u = 0$, satisfied weakly by the previous proposition, is satisfied pointwise
everywhere in $\Omega$.
\end{proof}

We can now give a new pointwise definition of derivative in the
direction of vector fields $X_1$  and $X_2$.

\begin{defi}\label{deriv}
Let $V$ be a Lipschitz vector field on $\Omega$ and let  $\xi_0\in \Omega$ and $\gamma(s)$ be a solution to
problem $\gamma' = V(\gamma) \, , \, \gamma(0)=\xi_0$. We say that a function $f\in \C^{\alpha}_{loc}(\Omega),$ with $\alpha
\in ]0,1[,$ has Lie derivative in the direction of the vector
field $V$ in $\xi_0$ if there exists
$$
\frac{d}{ds}(f\circ\gamma)\Big|_{s=0},
$$
and we will denote its value by
$Vf(\xi_0).$
\end{defi}

\noindent If the weak derivative of a function $f$ is sufficiently regular,
then the two notions of derivatives coincide. For the proof of the following
result see \cite[Proposition 5.2]{CM}.

\begin{prop} \label{lieweak}
If $f\in \C_{loc}^{\alpha}(\Omega)$ for some $\alpha\in ]0,1[$ and
its weak derivatives satisfy $X_{1,u}f\in \C_{loc}^{\alpha}(\Omega),
Y_uf\in L^p_{loc}(\Omega)$ with $p>1/\alpha,$ then for all
$\xi \in \Omega$ the Lie derivative $X_{1,u}f(\xi)$ exists and coincides
with the weak one.
\end{prop}

We are now ready to prove the main result concerning the foliation

\begin{proof}[Proof of corollary \ref{CORfoliation}]
By Theorem \ref{theoCregularity} we have $X^2u = 0$, and the previous proposition ensures that $X (Xu)$ is a Lie derivative, so by Definition \ref{deriv} we have
$$
\frac{d^2}{ds^2}u(\gamma(s)) = X^2u = 0\, .
$$
\end{proof}

\bibliographystyle{model1-num-names}

\end{document}